\documentclass[a4paper,11pt]{article}
\usepackage[utf8]{inputenc}
\usepackage[T1]{fontenc}
\usepackage{amsmath, amsthm, amssymb}
\usepackage{physics} 
\usepackage{braket} 
\usepackage{enumitem} 
\usepackage{esint} 
\usepackage{graphicx} 
\usepackage[toc,page]{appendix} 
\usepackage{hyperref} 

\newcommand{\mres}{\mathbin{\vrule height 1.6ex depth 0pt width 0.13ex\vrule height 0.13ex depth 0pt width 0.8ex}}

\newcommand{\footremember}[2]{%
    \footnote{#2}
    \newcounter{#1}
    \setcounter{#1}{\value{footnote}}%
}
\newcommand{\footrecall}[1]{%
    \footnotemark[\value{#1}]%
}
\newcommand\blfootnote[1]{%
    \begingroup
    \renewcommand\thefootnote{}\footnote{#1}%
    \addtocounter{footnote}{-1}%
    \endgroup
}

\theoremstyle{plain}
\newtheorem{thm}{Theorem}[section]
\newtheorem{prop}[thm]{Proposition}

\newtheorem{lem}[thm]{Lemma}
\newtheorem*{thm*}{Theorem}
\newtheorem*{prop*}{Proposition}
\newtheorem*{cor*}{Corollary}
\newtheorem*{lem*}{Lemma}

\theoremstyle{definition}
\newtheorem{defi}[thm]{Definition}

\newtheorem*{defi*}{Definition}
\newtheorem*{exa*}{Example}

\theoremstyle{remark}
\newtheorem{rmk}[thm]{Remark}
\newtheorem*{rmk*}{Remark}

\newcommand{\R}{\mathbf{R}}

\newcommand{\HH}{\mathcal{H}}
\newcommand{\LL}{\mathcal{L}}

\newcommand{\spt}{\mathrm{spt}}

\newcommand{\dm}{\,\mathrm{d}}
\usepackage[toc,page]{appendix} 

\title{The calibration method for the thermal insulation functional}

\author{C. Labourie\footremember{ucy}{University of Cyprus, Department of Mathematics \& Statistics, P.O. Box 20537, Nicosia, CY- 1678 CYPRUS.}\footnote{{\small \tt labourie.camille@ucy.ac.cy}} \and E. Milakis\footrecall{ucy}\footnote{{\small \tt emilakis@ucy.ac.cy}}\blfootnote{This work was co-funded by the European Regional Development Fund and the Republic of Cyprus through the Research and Innovation Foundation (Project: EXCELLENCE/1216/0025)}
}

\date{}

\begin{document}

\maketitle

\begin{abstract}
    We provide minimality criteria by construction of calibrations for
    functionals arising in the theory of Thermal Insulation.
\end{abstract}



\tableofcontents


\section{Introduction}

\subsection{Calibrations for free-discontinuity problems}

Free-discontinuity problems consists in minimizing the energy $E(u,K)$ of a pair composed of a function $u \in C^1(\R^n \setminus K)$ and a set $K$ of dimension $n - 1$.
The energy presents a competition between the Dirichlet energy of $u$ in $\R^n \setminus K$ and the surface energy of $K$.
The set $K$ is interpreted as an hypersurface (with possibly singularities) where $u$ jumps between different values.
This notion of pair composed of a function and of its discontinuity set is alternatively formalized by the space $SBV$ (special functions with bounded variations).
The model case of this kind of problem is the Mumford-Shah functional coming from image segmentation.

These problems generally present two kind of Euler-Lagrange equations.
Considering small perturbations of $u$ (with the discontinuity set $K$ being fixed), one obtains that $u$ satisfy a PDE with a boundary condition on each connected component of $\R^n \setminus K$.
As an example, minimizers of the Mumford-Shah functional are harmonic functions satisfying a Neumann boundary condition.
Considering small perturbations of $K$ under diffeormophisms, one obtains an equation that deals with the mean curvature of $K$.
However, these equations do not entirely characterise the minimizers.
We also point out that the minimizer may not be unique.
We summarize these difficulties as a lack of convexity of the functional $E$.

In~\cite{Alberti}, Alberti, Bouchitt\'e, Dal Maso have introduced a sufficient condition for minimality by adapting the calibration method that was known for minimal hypersurfaces.
Here is a simplified summary.
Let us consider a competitor $(u,K)$.
We define the complete graph $\Gamma_u$ of $u$ as the boundary of the subgraph of $u$.
It is the reunion of the graph of $u$ on $\R^n \setminus K$ and of vertical sides above $K$ (the discontinuities of $u$).
We let $\nu_{\Gamma_u}$ denote the normal vector to $\Gamma_u$ pointing into the subgraph of $u$.
A calibration for $(u,K)$ is a divergence-free vector field $\phi \colon \R^n \times \R \to \R^n \times \R$ such that
\begin{equation}\label{eq_calibration_principle1}
    E(u,K) = \int_{\Gamma_u} \! \phi \cdot \nu_{\Gamma_u} \mathrm{d}\HH^{n-1}
\end{equation}
and for all competitor $(v,L)$,
\begin{equation}\label{eq_calibration_principle2}
    E(v,L) \geq \int_{\Gamma_v} \! \phi \cdot \nu_{\Gamma_v} \mathrm{d}\HH^{n-1}.
\end{equation}
One observes that the Gauss-Green theorem and the divergence-free property imply
\begin{equation}
    \int_{\Gamma_u} \! \phi \cdot \nu_{\Gamma_u} \mathrm{d}\HH^{n-1} = \int_{\Gamma_v} \! \phi \cdot \nu_{\Gamma_v} \mathrm{d}\HH^{n-1}
\end{equation}
so the existence of such a vector field proves that $(u,K)$ is a minimizer.

In~\cite[Lemma 3.7]{Alberti}, the authors provide four conditions to ensure the properties (\ref{eq_calibration_principle1}), (\ref{eq_calibration_principle2}) above and it is convenient to take these conditions as a definition of calibrations.
Minimality criterias for the Mumford-Shah functional are proved by construction of calibrations in~\cite{Alberti},~\cite{M1},~\cite{M2},~\cite{M3},~\cite{M4}.
The principle of calibrations also inspired a fast primal-dual algorithm to minimize the Mumford-Shah functional (\cite{C1},\cite{C2}).

In general, we don't know if calibrations exist for minimizers of this kind of problem.
This question is related to the non-existence of a duality gap (see~\cite[Section 3.13]{Alberti},~\cite{C1} and also~\cite{Bo}).
There is no general recipe to follow and the construction can be very difficult.
The crack-tip is famous example of minimizer of the Mumford-Shah functional for which a calibration has not been found.

\subsection{The thermal insulation functional}

A free boundary problem related to thermal insulation was recently studied by Caffarelli--Kriventsov (\cite{CK},~\cite{K}) and also Bucur--Luckhaus--Giacomini (\cite{BL},~\cite{BG}).
Relaxing the problem in $SBV$, it consists in minimizing
\begin{equation*}
    E(u) = \int \! \abs{\nabla u}^2 \dm x + \beta \int_{J_u} \! (u^-)^2 + (u^+)^2 \dm \HH^{n-1} + \gamma^2 \LL^n(\set{u > 0}),
\end{equation*}
where $\beta, \gamma > 0$ and the competitors are functions $u \in SBV(\R^n)$ such that $u = 1$ on a given bounded open set $\Omega \subset \R^n$ and $0 \leq u \leq 1$ on $\R^n$.
Here $J_u$ is the set of all jump points of $u$, that is the points $x$ for which there exist two real numbers $u^- < u^+$ and a (unique) vector $\nu_u(x) \in \mathbf{S}^{n-1}$ such that
\begin{subequations}
    \begin{align}
    & \lim_{r \to 0} \fint_{B_r \cap H^+} \! \abs{u(y) - u^-} \dm y = 0 \\
    & \lim_{r \to 0} \fint_{B_r \cap H^-} \! \abs{u(y) - u^+} \dm y = 0,
    \end{align}
\end{subequations}
where
\begin{subequations}
    \begin{align}
        H^+ & = \set{y \in \R^n | (y - x) \cdot \nu_u(x) > 0} \\
        H^- & = \set{y \in \R^n | (y - x) \cdot \nu_u(x) < 0}.
    \end{align}
\end{subequations}
The function $u$ can be interpreted as the temperature (which is fixed to $1$ on $\Omega$) and $J_u$ as an isolating layer which has no width and no thermal conductivity.


Caffarelli--Kriventsov and Bucur--Luckhaus--Giacomini have shown the existence of minimizers in~\cite{CK} and~\cite{BL},\cite{BG}.
They prove a non-degeneracy property: there exists $0 < \delta < 1$ (depending only on $n$ and $\Omega$) such that such that $\spt(u) \subset B(0,\delta^{-1})$ and
\begin{equation}
    u \in \set{0} \cup [\delta,1] \quad \text{$\LL^n$-a.e.\ on $\R^n$}.
\end{equation}
They also prove that the jump set $J_u$ is essentially closed, $\HH^{n-1}(\overline{J_u} \setminus J_u) = 0$, and that it satisfies uniform density estimates.
In~\cite{K}, it was proven that the jump set is locally the union of the graphs of two $C^{1,\alpha}$ functions provided that it is trapped between two planes which are sufficiently close.
The approaches in these papers highlighted the similarities and the differences between the thermal insulation problem and the Mumford-Shah functional.
In particular, for the thermal insulation problem, one has to deal with an harmonic function satisfying a Robin boundary condition at the boundary rather than a Neumann boundary condition.

In~\cite{LM}, we have shown the higher integrability of the gradient for minimizers of the thermal insulation problem, an analogue of De Giorgi's conjecture for the Mumford-Shah functional, and deduced that the singular part of the free boundary has Hausdorff dimension strictly less than $n - 1$.
Variants of this problem have been studied in~\cite{B1},~\cite{B2},~\cite{B3},~\cite{B4}.
A numerical implementation has been proposed in~\cite{BF}.

\subsection{Summary of results}

The purpose of this article is to provide minimality criteria for the thermal insulation functional by construction of calibrations.
Besides minimality criteria, we are also simply interested in understanding what calibrations look like for this kind of problem.
The article is divided into two parts.
In the first part, we fix an open set $A$ of $\R^n$ and we consider the homogeneous functional
\begin{equation}
    E_0(u) = \int_A \! \abs{\nabla u}^2 \dm x + \beta \int_{J_u \cap A} \! (u^+)^2 + (u^-)^2 \dm \HH^{n-1},
\end{equation}
where $u \in SBV(A)$.
In Theorem~\ref{thm_harmonic}, we present a sufficient condition so that a non-negative harmonic function $u \colon A \to \R$ is a minimizer of $E_0$ among all competitors $v \in SBV(A)$ with $\set{v \ne u} \subset \subset A$.
The condition is also necessary in dimension one but not in higher dimension.
An anologous questions had been studied for the homogeneous Mumford-Shah functional
\begin{equation}
    F_0(u,A) = \int_A \! \abs{\nabla u}^2 \dm x + \beta \HH^{n-1}(J_u)
\end{equation}
by Chambolle without calibrations (\cite[Theorem 3.1(i)]{F}) and Alberti, Bouchitt\'e, Dal Maso with calibrations (\cite[Section 4.6]{Alberti}).

In the second part, we come back to the full thermal insulation functional,
\begin{equation}
    E(u) = \int \! \abs{\nabla u}^2 \dm x + \beta \int_{J_u} \! (u^+)^2 + (u^-)^2 \dm \HH^{n-1} + \gamma^2 \LL^n(\set{u > 0})
\end{equation}
where $u \in SBV(\R^n)$ is such that $u = 1$ on a given bounded open set $\Omega$.
The section starts with an informal discussion about the case $\Omega = B(0,1)$.
The relevant competitors should be either the indicator function $\mathbf{1}_\Omega$ or an harmonic function supported in a bigger ball.
We can make explicit computations with these competitors to find minimality criterias.
Then we prove and generalize these criterias to other domains $\Omega$ using calibrations.
In Theorem~\ref{thm_indicator1} and~\ref{thm_indicator2}, we prove two sufficient conditions which imply that $\mathbf{1}_\Omega$ is a minimizer.
In Theorem~\ref{thm_u}, we come back back to the case $\Omega = B(0,1)$ and prove a sufficient condition so that an harmonic function supported in a bigger ball is a minimizer.

\textbf{Acknowledgement}. We would like to thank Antoine Lemenant and Antonin Chambolle for helpful discussions about extensions of normal vector fields.

\section{The homogenous functional}

\subsection{Statement of the problem and calibrations}

We fix a parameter $\beta > 0$.
Given a Borel set $S \subset \R^n$ and a $SBV$ function $u$ in a neighborhood of $S$, we define
\begin{equation}
    E_0(u,S) = \int_S \! \abs{\nabla u}^2 \dm x + \beta \int_{J_u \cap S} \! (u^+)^2 + (u^-)^2 \dm \HH^{n-1}.
\end{equation}

Let $V$ be an open set and let $u \colon V \to \R$ be a a non-negative harmonic function.
We are interested in finding a sufficient condition so that for all $v \in SBV(V)$ with $\set{v \ne u} \subset \subset V$, we have
\begin{equation}
    E_0(u,V) \leq E_0(v,V).
\end{equation}
It is equivalent to require that for all open set $A \subset \subset V$, for all $v \in SBV(V)$ with $v = u$ in $V \setminus \overline{A}$, we have
\begin{equation}
    E_0(u,\overline{A}) \leq E_0(v,\overline{A}).
\end{equation}

\begin{thm}\label{thm_harmonic}
    Let $A$ be a bounded open set of $\R^n$ with Lipschitz boundary
    Let $u \colon A \to \R$ be a non-negative harmonic function which has a $C^1$ extension in open set $V$ containing $\overline{A}$.
    We assume that there exists $0 \leq m \leq M$ such that $m \leq u \leq M$ in $A$ and
    \begin{equation}\label{eq_harmonic_condition}
        (M - m)^2 - (M - \delta)^2 \leq \beta_0 (m^2 + \delta^2),
    \end{equation}
    where
    \begin{equation}
        \beta_0 = \beta \frac{(M - m)}{\sup_A \abs{\nabla u}} \qquad \text{and} \qquad \delta = \frac{M}{1 + \beta_0}.
    \end{equation}
    Then $u$ is a minimizer of $E_0(v,\overline{A})$ among all $v \in SBV(V)$ such that $v = u$ in $V \setminus \overline{A}$.
\end{thm}

\begin{rmk}
    The condition (\ref{eq_harmonic_condition}) can be written
    \begin{equation}\label{eq_harmonic_condition2}
        \int_m^{\delta} \! 2 (M - t) \dm t \leq \beta_0 (m^2 + \delta^2),
    \end{equation}
    to highlight its similarity with the condition $\abs{\int_r^s \! \phi^x(x,t) \dm t} \leq \beta (r^2 + s^2)$ in the Definition~\ref{defi_harmonic_calibration} below (the definition of calibrations).
    The function
    \begin{equation}
        s \mapsto \beta_0 (m^2 + s^2) - \int_m^s \! 2 (M - t) \dm t
    \end{equation}
    attains its minimum at $s = \frac{M}{1 + \beta_0}$ so (\ref{eq_harmonic_condition2}) is also equivalent to say that for all $s \in [0,M]$,
    \begin{equation}
        \int_m^s \! 2 (M - t) \dm t \leq \beta_0 (m^2 + s^2).
    \end{equation}
\end{rmk}

\begin{rmk}
    Let us consider the case where $n = 1$, $A = ]0,h[$ (where $h > 0$) and $u$ is an affine function whose graph joins $(0,m)$ and $(h,M)$.
    Then the theorem condition is necessary.
    In this cases $\beta_0 = \beta h$, $\delta = \frac{M}{1 + \beta h}$ and the condition (\ref{eq_harmonic_condition}) amounts to
    \begin{equation}
        h^{-1} (M - m)^2 \leq \beta (m^2 + \delta^2) + h^{-1} (M - \delta)^2.
    \end{equation}
    One recognizes that the right-hand side is the Dirichlet energy of $u$.
    The left-hand side is the energy of the jump function whose graph joins $(0,m)$, $(0,\delta)$ and $(h,M)$.
\end{rmk}

\begin{rmk}
    The condition (\ref{eq_harmonic_condition}) is trivially satisfied if $m \geq \delta$, which can be simplified to
    \begin{equation}
        m \geq \beta^{-1} \sup_A \abs{\nabla u}.
    \end{equation}
\end{rmk}

\begin{rmk}
    For the homogeneous Mumford-Shah functional,
    \begin{equation}
        F_0(u) = \int_A \! \abs{\nabla u}^2 + \beta \HH^{n-1}(J_u),
    \end{equation}
    Chambolle found the condition $(M - m) \sup_A \abs{\nabla u} \leq \beta$ (see~\cite[Theorem 3.1 (i)]{F}).
    This is somewhat analogous to our condition because this can be rewritten
    \begin{equation}
        \int_m^M \! 2 (M - t) \dm t \leq \beta_0
    \end{equation}
    where $\beta_0 = \beta \frac{(M - m)}{\sup_A \abs{\nabla u}}$.
\end{rmk}

We are going to state the notion of calibrations associated to our problem.
The existence of such a vector field implies that $u$ is a minimizer of $E_0(v,\overline{A})$ among $v \in SBV(V)$ such that $v = u$ in $V \setminus \overline{A}$.
This is a consequence of~\cite[Section 2 and 3]{Alberti} but we provide a detailed explanation in Appendix~\ref{appendix_calibration}.
We underline that our Dirichlet problem is different from the one of~\cite[Definition 3.1]{Alberti} because our boundary condition is one-sided (a competitor $v$ might jump on $\partial A$).
This is why we define $\phi$ up to $\partial A \times [0,M]$.

\begin{defi}\label{defi_harmonic_calibration}
    Let $A$ be a bounded open set of $\R^n$ with Lipschitz boundary, let $V$ be an open set containing $\overline{A}$ and let $u \in SBV(V)$ be such that $0 \leq u \leq M$ on $A$ (for some $M > 0$).
    A calibration for $u$ in $\overline{A} \times [0,M]$ is a Borel map
    \begin{equation}
        \phi = (\phi^x,\phi^t) \colon \overline{A} \times [0,M] \to \R^n \times \R
    \end{equation}
    which is bounded and approximately regular on $\overline{A} \times [0,M]$, divergence-free on $A \times ]0,M[$ in the sense of distribution and such that
    \begin{enumerate}
        \item[(a)] $\phi^t(x,t) \geq \tfrac{1}{4} \abs{\phi^x(x,t)}^2$ for $\LL^n$-a.e.\ $x \in A$ and every $t \in [0,M]$;
        \item[(b)] $\abs{\int_r^s \! \phi^x(x,t) \dm t} \leq \beta (r^2 + s^2)$ for $\HH^{n-1}$-a.e.\ $x \in \overline{A}$ and every $r,s \in [0,M]$;
        \item[(a')] $\phi^x(x,u) = 2 \nabla u$ and $\phi^t(x,u) = \abs{\nabla u}^2$ for $\LL^n$-a.e.\ $x \in \overline{A}$;
        \item[(b')] $\int_{u^-}^{u^+} \! \phi^x(x,t) \dm t = \beta \left[(u^-)^2 + (u^+)^2\right]\nu_u$ for $\HH^{n-1}$-a.e.\ $x \in J_u \cap \overline{A}$.
    \end{enumerate}
\end{defi}

With regard to Theorem~\ref{thm_harmonic}, the candidate function $u$ is smooth so we have $J_u = \emptyset$ and we don't need to check (b').
We conclude this section with two miscellaneous remarks.

\begin{rmk}[Scaling]
    We detail the scaling properties of $E_0$.
    We write $E_0(u,\beta,A)$ to explicit the parameters $\beta$ in the definition of $E_0$.
    For all $M \in \R$, we have
    \begin{equation}
        E_0(M u, \beta, A) = M^2 E_0(u,\beta, A),
    \end{equation}
    and for all $h > 0$,
    \begin{equation}
        E_0(u_h, \beta h, h^{-1}A) = h^{2 - n} E_0(u,\beta,A).
    \end{equation}
    where $u_h(x) = u(h \cdot x)$.
    Thus, for all $h > 0$ and $M \geq 0$,
    \begin{equation}
        E_0(M^{-1} u_h, \beta h, h^{-1}A) = M^{-2} h^{2 - n} E_0(u,\beta,A).
    \end{equation}
\end{rmk}



\begin{rmk}[Slope along a jump]\label{rmk_jump}
    This remark is an example of application of~\cite[Lemma 2.5]{Alberti} that we will use repeatedly in the next constructions.
    We consider three $C^1$ functions $\sigma, u,v \colon \R^n \to \R$.
    We work in $\R^n \times \R$ and we introduce the hypersurface
    \begin{equation}
        H   = \set{(x,t) \in \R^n \times \R | t = \sigma(x)}
    \end{equation}
    and the vector field
    \begin{equation}
        \phi =
        \begin{cases}
            (2 \nabla u,\abs{\nabla u}^2) & \text{for $t > \sigma(x)$} \\
            (2 \nabla v,\abs{\nabla v}^2) & \text{for $t < \sigma(x)$}.
        \end{cases}
    \end{equation}
    A normal vector field to $H$ is $(-\nabla\sigma,1)$ and we have along $H$,
    \begin{equation}
        \phi(x,\sigma(x)^+)
        \cdot
        \begin{pmatrix} -\nabla\sigma \\ 1 \end{pmatrix}
        =
        \phi(x,\sigma(x)^-)
        \cdot
        \begin{pmatrix} -\nabla\sigma \\ 1 \end{pmatrix}.
    \end{equation}
    if and only if $\nabla \sigma = \tfrac{1}{2} (\nabla u + \nabla v)$.
    Therefore, $\phi$ is divergence-free in the sense of distributions provided that $u,v$ are harmonic and $\sigma = \tfrac{1}{2} (u + v)$ (modulo an additive constant).
\end{rmk}

\subsection{The one dimensional case}\label{section_harmonic_1d}

\subsubsection{A short analysis}

We fix constants $a < b$ and $0 \leq m \leq M$.
We consider the functional
\begin{equation}
    E_0(u) = \int \! \abs{\nabla u}^2 \dm x + \beta \int_{J_u} \! (u^+)^2 + (u^-)^2 \dm \HH^{n-1}
\end{equation}
defined over the function $u \in SBV(\R)$ such that $u = m$ on $]-\infty,a[$ and $u = M$ on $]b,\infty]$.
By scaling, it suffices to study the case $[a,b] = [0,1]$.

We are looking for the minimizer(s), depending on $m$ and $M$.
Here is a summary (without details).
In the case $m = M$, the affine (constant) competitor is the only minimizer and from now on, we assume $m < M$.
First, we try to find the jump competitors which have the least energy.
It is never convenient to do more than one jump.
At each side of the jump, the optimal value of the function corresponds to the Robin condition $\partial_\nu u = \beta u$ (where $\nu$ is the inward unit normal vector to the side).
The optimal location of the jump may be $x_0 = 0$ or some $x_0 \in ]0,1[$.
It cannot be $x_0 = 1$ though and this comes from the fact that $m < M$.
If $m < \frac{M}{1 + \beta}$, the best location is $x_0 = 0$ and the jump is going from $m$ to $\frac{M}{1 + \beta}$.
If $m > \frac{M}{1 + \beta}$, the best location is a certain $x_0 \in ]0,1[$ but then the affine competitor has necessarily a smaller energy.
We conclude that the affine competitor is a minimizer if and only if it is better than the jump at $x = 0$, that is
\begin{equation}\label{eq_harmonic_1d_condition}
    (M - m)^2 \leq \beta m^2 + \left(M - \frac{M}{1 + \beta}\right)^2.
\end{equation}
We define $\delta = \frac{M}{1 + \beta}$ and we observe that (\ref{eq_harmonic_1d_condition}) is equivalent to
\begin{equation}\label{eq_harmonic_1d_condition2}
    \int_m^\delta \! 2 (M - t) \dm t \leq \beta (m^2 + \delta^2).
\end{equation}

Now, we try to guess the calibration in the limit case
\begin{equation}\label{eq_harmonic_1d_limit_case}
    \int_m^\delta \! 2 (M - t) \dm t = \beta (m^2 + \delta^2).
\end{equation}
Notice that we have necessarily $m < \delta$.
There are two minimizers; the first one is the affine function $u$ whose graph joins the points $(0,m)$, $(1,M)$ and the second one is the jump function $v$ whose graph joins the points $(0,m)$, $(0,\delta)$, $(1,M)$.
A nice property of calibrations is that they calibrate all minimizers simultaneously (see Remark~\ref{rmk_calibration_equality} or~\cite[Remark 3.6]{Alberti}).
We should have $\phi = (2\nabla u,\abs{\nabla u}^2)$ on the graph of $u$, that is $\phi = (2 (M - m), (M - m)^2)$ on $t = m + (1 - m)x$.
And we should have as well $\phi = (2 \nabla v, \abs{\nabla v}^2)$ on the graph of $v$, that is $\phi = (2 (M - \delta), (M - \delta)^2)$ on $t = m + (M - \delta)x$.
Finally, we should have
\begin{equation}
    \int_m^\delta \! \phi^x(0,t) \dm t = \beta (m^2 + \delta^2).
\end{equation}
A simple solution is to set
\begin{equation}
    \phi = \left(\frac{2 (M - t)}{1 - x}, \left[\frac{M - t}{1 - x}\right]^2\right)
\end{equation}
for $t \geq m + (1 - m)x$.
Next, we try to determine $\phi^x(0,t)$ for $t \in [0,m]$.
We have necessarily
\begin{align}
    \int_0^m \! \phi^x(0,t) \dm t & = \int_0^\delta \! \phi^x(0,t) \dm t - \int_m^\delta \! \phi^x(0,t) \dm t \\
                                  & \leq \beta \delta^2 - \beta (m^2 + \delta^2)                              \\
                                  & \leq - \beta m^2.
\end{align}
We suggest to set $\phi^x(0,t) = -2 \beta t$ on $[0,m]$.
Then we try to extend $\phi$ in a simple way while respecting the requirements of calibrations and we arrive at
\begin{equation}
    \phi =
    \begin{cases}
        \left(\frac{2 (M - t)}{1 - x},\left[\frac{M - t}{1 - x}\right]^2\right) & \text{if $m + \tau x \leq t \leq M$}            \\
        \left(2 (M - m),\left[M - m\right]^2\right)                             & \text{if $m + \sigma x \leq t \leq m + \tau x$} \\
        \left(-2 \beta m,\beta^2 m^2\right)                                     & \text{if $m \leq t \leq m + \sigma x$}          \\
        \left(-2 \beta t,\beta^2 m^2\right)                                     & \text{if $0 \leq t \leq m$.}
    \end{cases}
\end{equation}
where $\tau = M - m$ and $\sigma = \frac{1}{2} ((M - m) - \beta m)$.
The slope $\sigma$ has been chosen as in Remark~\ref{rmk_jump}.
In the next section, we generalize this construction when there is only an inequality in (\ref{eq_harmonic_1d_limit_case}).
\begin{figure}[ht]
    \begin{center}
        \includegraphics[width = \linewidth]{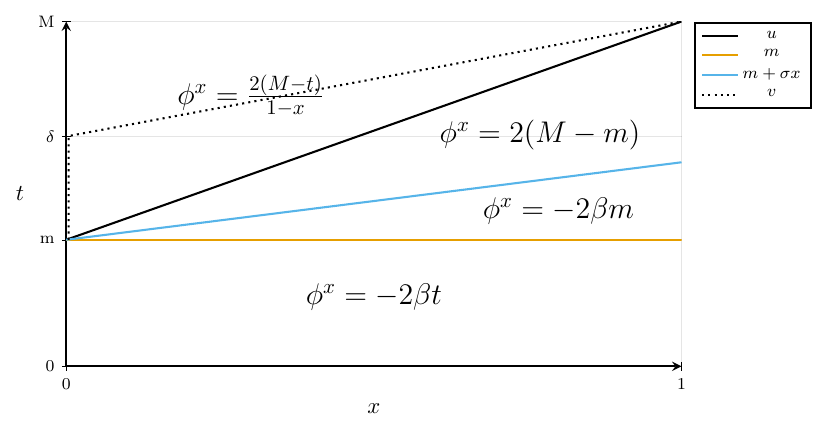}
    \end{center}
\end{figure}

\subsubsection{The construction}

\textbf{Context}.
Let $0 \leq m \leq M$ and let $u \colon [0,1] \to \R$ be the affine function such that $u(0) = m$ and $u(1) = M$.
We assume that
\begin{equation}
    \int_m^\delta \! 2 (M - t) \dm t \leq \beta (m^2 + \delta^2),
\end{equation}
where $\delta = \frac{M}{1 + \beta}$.
We build a calibration for $u$ in $[0,1] \times [0,M]$.

The construction brings into play an intermediary constant $0 \leq \lambda \leq \beta$ such that
\begin{equation}
    \int_m^\delta \! 2 (M - t) \dm t \leq \lambda m^2 + \beta \delta^2.
\end{equation}
We will also need the conditions $\lambda m \leq M - m$ and $\lambda m \leq \frac{\beta M}{1 + \beta}$.
If $m \leq \delta$, these two conditions follows from the fact that $\lambda \leq \beta$.
Otherwise, it suffices to take $\lambda = 0$ for example.
Finally, we define
\begin{equation}
    \phi =
    \begin{cases}
        \left(\frac{2 (M - t)}{1 - x},\left[\frac{M - t}{1 - x}\right]^2\right) & \text{if $m + \tau x \leq t \leq M$}            \\
        \left(2 (M - m),\left[M - m\right]^2\right)                             & \text{if $m + \sigma x \leq t \leq m + \tau x$} \\
        \left(-2 \lambda m,\lambda^2 m^2\right)                                 & \text{if $m \leq t \leq m + \sigma x$}          \\
        \left(-2 \lambda t,\lambda^2 m^2\right)                                 & \text{if $0 \leq t \leq m$.}
    \end{cases}
\end{equation}
where $\tau = M - m$ and $\sigma = \frac{1}{2} ((M - m) - \lambda m)$.
Notice that we have $0 \leq \sigma \leq \tau$ because $0 \leq \lambda m \leq M - m$.
The slope $\sigma$ has been chosen as in Remark~\ref{rmk_jump}.

We prove that for all $x$ and all $r \leq s$,
\begin{equation}
    \abs{\int_r^s \! \phi^x(x,t) \dm t} \leq \beta (r^2 + s^2).
\end{equation}
We have globally
\begin{equation}
    \phi^x \geq -2 \lambda t \geq -2 \beta t
\end{equation}
so it suffices to control $\int_r^s \! \phi^x \dm t$ from above.
Let us fix $t$.
We are going to see that the function
\begin{equation}
    r \mapsto \beta (r^2 + s^2) - \int_r^s \! \phi^x(x,t) \dm t
\end{equation}
is non-decreasing on $[0,M]$.
If $0 \leq r \leq m$,
\begin{equation}
    \beta (r^2 + s^2) - \int_r^s \! \phi^x(x,t) \dm t = \beta r^2 - \lambda r^2 + (\ldots)
\end{equation}
where $(\ldots)$ does not depend on $r$.
The right-hand side is non-decreasing on $[0,m]$ because $\lambda \leq \beta$.
If $m \leq r \leq m + \sigma x$,
\begin{equation}
    \beta (r^2 + s^2) - \int_r^s \! \phi^x(x,t) \dm t = \beta r^2 - 2 \lambda m r + (\ldots)
\end{equation}
where $(\ldots)$ does not depend on $r$.
The right-hand side is non-decreasing on $[m,m + \sigma x]$ because $\lambda \leq \beta$.
Finally, the function in non-decreasing on $[m + \sigma x, M]$ because $\phi^x$ is non-negative on this interval.
We conclude that it suffice to show that for all $x$ and $t$,
\begin{equation}
    \int_0^s \! \phi^x(x,t) \dm t \leq \beta s^2.
\end{equation}
This is trivial if $s \leq m$ so we assume $s > m$.
There exists $x_0$ such that $s = m + (M - m)x_0$.
The function
\begin{equation}
    x \mapsto \int_r^s \! \phi^x(x,t) \dm t
\end{equation}
is non-decreasing on $[x_0,1]$ because the part where $\phi^x \leq 0$ gets bigger at the expense of the part where $\phi^x \geq 0$.
Now, it suffices to show that for all $x$ and $m + (M - m)x \leq s \leq M$,
\begin{equation}
    \int_0^s \! \phi^x(x,t) \dm t \leq \beta s^2.
\end{equation}
We apply the divergence theorem (\cite[Lemma 2.4]{Alberti}) in the polygone delimited by the cycle $(x,0), (x,s), (1,M), (0,1), (0,0)$.
We get
\begin{equation}
    \int_0^s \! \phi^x(x,t) \dm t = -\frac{(M - s)^2}{1 - x} + (M - m)^2 - \lambda m^2 + \lambda^2 m^2 x
\end{equation}
so
\begin{equation}
    \beta s^2 - \int_0^s \! \phi^x(x,t) \dm t = \beta s^2 + \frac{(M - s)^2}{1 - x} - \lambda^2 m^2 x + \lambda m^2 - (M - m)^2.
\end{equation}
The function $s \mapsto \beta s^2 + \frac{(M - s)^2}{1 - x}$ is minimal for $s = \frac{M}{1 + \beta (1 - x)}$ and its the minimum value is $\frac{\beta M^2}{1 + \beta (1 - x)}$.
We have therefore
\begin{equation}
    \beta s^2 - \int_0^s \! \phi^x(x,t) \dm t \geq \frac{\beta M^2}{1 + \beta (1 - x)} - \lambda^2 m^2 x + \lambda m^2 - (M - m)^2.
\end{equation}
The function $x \mapsto \frac{\beta M}{1 + \beta (1 - x)} - \lambda^2 m^2 x$ is non-decreasing on $[0,1]$ because $\lambda m \leq \frac{\beta M}{1 + \beta}$.
We are thus led to
\begin{equation}
    \beta s^2 - \int_0^s \! \phi^x(x,t) \dm t \geq \frac{\beta M^2}{1 + \beta} + \lambda m^2 - (M - m)^2
\end{equation}
As $\frac{\beta M^2}{1 + \beta} = \beta \delta^2 + (M - \delta)^2$, the previous quantity can be rewritten
\begin{equation}
    \lambda m^2 + \beta \delta^2 - \int_m^\delta \! 2 (M - t) \dm t
\end{equation}
and this is non-negative by assumption.

\subsection{Proof of Theorem~\ref{thm_harmonic}}

\textbf{Context}.
Let $A$ be a bounded open set of $\R^n$ with Lipschitz boundary.
Let $u \colon A \to \R$ be a non-negative harmonic function which has $C^1$ extension in a neighborhood of $\overline{A}$.
We assume that for some $0 \leq m \leq M$, we have $m \leq u \leq M$ in $A$ and
\begin{equation}
    \int_m^{\delta} \! 2 (M - t) \dm t \leq \beta_0 (m^2 + \delta^2),
\end{equation}
where
\begin{equation}
    \beta_0 = \beta \frac{(M - m)}{\sup_A \abs{\nabla u}} \qquad \text{and} \qquad \delta = \frac{M}{1 + \beta_0}.
\end{equation}
We build a calibration for $u$ in $\overline{A} \times [0,M]$.

The construction brings into play an intermediary constant $0 \leq \lambda \leq \beta_0$ such that
\begin{equation}
    \int_m^{\delta} \! 2 (M - t) \dm t \leq \lambda m^2 + \beta_0 \delta^2.
\end{equation}
We also need the conditions $\lambda m \leq M - m$ and $\lambda m \leq \frac{\beta_0 M}{1 + \beta_0}$.
If $m \leq \delta$, these two conditions follow from the fact that $\lambda \leq \beta_0$.
Otherwise, it suffices to take $\lambda = 0$ for example.
Finally, we define
\begin{equation}
    \phi = \left(\varphi^x \cdot \frac{\nabla u(x)}{M - m}, \varphi^t \cdot \abs{\frac{\nabla u(x)}{M - m}}^2\right),
\end{equation}
where
\begin{equation}
    \varphi =
    \begin{cases}
        \left(2 (M - m)\frac{M - t}{M - u},\left[(M - m)^2\frac{M - t}{M - u}\right]^2\right) & \text{if $u(x) \leq t \leq M$}         \\
        \left(2 (M - m),\left[M - m\right]^2\right)                                           & \text{if $\sigma(x) \leq t \leq u(x)$} \\
        \left(-2 \lambda m,\lambda^2 m^2\right)                                               & \text{if $m \leq t \leq \sigma(x)$}    \\
        \left(-2 \lambda t,\lambda^2 m^2\right)                                               & \text{if $0 \leq t \leq m$}.
    \end{cases}
\end{equation}
and
\begin{align}
    \sigma(x) & = m + \frac{1}{2}(u - m)(1 - \tfrac{\lambda m}{M - m}) \\
              & = m + \frac{1}{2} \frac{u - m}{M - m}(M - m - \lambda m).
\end{align}
As $0 \leq \lambda m \leq M - m$, we see that $m \leq \sigma(x) \leq u(x)$.
We prove that for all $x$ and $r \leq s$,
\begin{equation}
    \abs{\int_r^s \! \varphi^x(x,t) \dm t} \leq \beta_0 (r^2 + s^2).
\end{equation}
This is a minor variant of the one-dimensional case.
We have globally
\begin{equation}
    \varphi^x \geq -2 \lambda t \geq -2 \beta_0 t
\end{equation}
so it suffices to control $\int_r^s \! \varphi^x \dm t$ from above.
Let us fix $t$.
We are going to see that the function
\begin{equation}
    r \mapsto \beta_0 (r^2 + s^2) - \int_r^s \! \varphi^x(x,t) \dm t
\end{equation}
is non-decreasing on $[0,M]$.
If $0 \leq r \leq m$,
\begin{equation}
    \beta_0 (r^2 + s^2) - \int_r^s \! \varphi^x(x,t) \dm t = \beta_0 r^2 - \lambda r^2 + (\ldots)
\end{equation}
where $(\ldots)$ does not depend on $r$.
The right-hand side is non-decreasing on $[0,m]$ because $\lambda \leq \beta_0$.
If $m \leq r \leq \sigma(x)$,
\begin{equation}
    \beta_0 (r^2 + s^2) - \int_r^s \! \varphi^x(x,t) \dm t = \beta_0 r^2 - 2 \lambda m r + (\ldots)
\end{equation}
where $(\ldots)$ does not depend on $r$.
The right-hand side is non-decreasing on $[m,\sigma(x)]$ because $\lambda \leq \beta_0$.
Finally, the function in non-decreasing on $[u(x),M]$ because $\varphi^x$ is non-negative on this interval.
We conclude that it suffice to show that for all $x$ and $s$,
\begin{equation}
    \int_0^s \! \varphi^x(x,t) \dm t \leq \beta_0 s^2.
\end{equation}
This is trivial if $s \leq m$ so we assume $s > m$.
There exists $x_0$ such that $s = u(x_0)$.
Since $\sigma \leq u$, we have in particular $\sigma(x_0) \leq s$.
For $x$ such that $\sigma(x) \leq s \leq u(x)$, we compute
\begin{align}
    \int_0^s \! \varphi^x(x,t) \dm t  & = -2 \lambda m^2 - 2 \lambda m (\sigma - m) + 2 (M - m) (s - \sigma) \\
    \begin{split}
                                      & = -2 \lambda m - 2 \lambda m (\sigma - m) + 2 (M - m) (s - m)        \\
                                      & \hspace{117pt} + 2 (M - m) (m - \sigma)
    \end{split}                      \\
    \begin{split}
                                      & = -2 \lambda m^2 - 2 (\sigma - m) (M - m + \lambda m)                \\
                                      & \hspace{117pt} + 2 (M - m) (s - m).
    \end{split}
\end{align}
As $u(x_0) \leq u(x)$ and $0 \leq \lambda m \leq M - m$, we have $\sigma(x_0) \leq \sigma(x)$ and it follows that
\begin{equation}
    \int_0^s \! \varphi^x(x,u) \leq \int_0^s \! \varphi^x(x_0,u) \dm.
\end{equation}
Now, it suffices to show that for all $x$ and $u(x) \leq s \leq m$,
\begin{equation}
    \int_0^s \! \varphi^x(x,t) \dm t \leq \beta_0 s^2.
\end{equation}
We compute
\begin{multline}
    \int_0^s \! \varphi^x(x,t) \dm t = - \lambda m^2 - 2 \lambda m (\sigma - m) + 2 (M - m) (u - \sigma) \\ + \frac{M - m}{M - u}\left((M - u)^2 - (M - s)^2\right).
\end{multline}
By definition
\begin{align}
    \sigma - m & = \frac{1}{2}\frac{u - m}{M - m} (M - m - \lambda m) \\
    u - \sigma & = \frac{1}{2}\frac{u - m}{M - m} (M - m + \lambda m)
\end{align}
so
\begin{equation}
    - 2 \lambda m (\sigma - m) + 2 (M - m) (u - \sigma) = \frac{u - m}{M - m} ((M - m)^2 + \lambda^2 m^2).
\end{equation}
We also write
\begin{align}
    (M - m) (M - u) & = (M - m)^2 + (M - m) (m - u) \\
                    & = (M - m)^2 - (M - m)^2 \frac{u - m}{M - m}
\end{align}
and we arrive at
\begin{equation}
    \int_0^s \! \varphi^x(x,t) \dm t = - \lambda m^2 + \lambda^2 m^2 \frac{u - m}{M - m} + (M - m)^2 - \frac{M - m}{M - u} (M - s)^2.
\end{equation}
Therefore
\begin{multline}
    \beta_0 s^2 - \int_0^s \! \varphi^x(x,t) \dm t = \beta_0 s^2 + \frac{M - m}{M - u} (M - t)^2 \\ - \lambda^2 m^2 \frac{u - m}{M - m} + \lambda m^2 - (M - m)^2.
\end{multline}
For a fixed $x$, the function $s \mapsto \beta_0 s^2 + \frac{M - m}{M - u} (M - s)^2$ is minimal for
\begin{equation}
    s = \frac{M}{1 + \beta_0 \frac{M - u}{M - m}}
\end{equation}
and its minimum value is
\begin{equation}
    \frac{\beta_0 M^2}{1 + \beta_0 \frac{M - u}{M - m}}.
\end{equation}
We have therefore
\begin{multline}
    \beta_0 s^2 - \int_0^s \! \varphi^x(x,t) \dm t \\ \geq \frac{\beta_0 M^2}{1 + \beta_0 \frac{M - u}{M - m}} - \lambda^2 m^2 \frac{u - m}{M - m} + 2 \lambda m^2 - (M - m)^2.
\end{multline}
The function
\begin{equation}
    x \mapsto \frac{\beta_0 M^2}{1 + \beta_0 \frac{M - x}{M - m}} - \lambda^2 m^2 \frac{x - m}{M - m}
\end{equation}
is non decreasing on $[m,M]$ because $\lambda m \leq \frac{\beta_0 M}{1 + \beta_0}$.
We are thus led to
\begin{equation}
    \beta_0 s^2 - \int_0^s \! \varphi^x(x,t) \dm t \geq \frac{\beta_0 M^2}{1 + \beta_0} + \lambda m^2 - (M - m)^2
\end{equation}
As $\frac{\beta_0 M^2}{1 + \beta_0} = \beta_0 \delta^2 + (M - \delta)^2$, the right-hand side can be rewritten
\begin{equation}
    \lambda m^2 + \beta_0 \delta^2 - \int_m^{\delta} \! 2 (M - t) \dm t.
\end{equation}

\section{The full functional}

\subsection{Definition of the problem and calibrations}

We fix a bounded open set $\Omega \subset \R^n$ with Lipschitz boundary.
We fix parameters $\beta > 0$ and $\gamma > 0$.
Given $u \in SBV(\R^n)$ such that $u = 1$ on $\Omega$ and $0 \leq u \leq 1$ in $\R^n$, we define
\begin{equation}\label{eq_E}
    E(u) = \int \! \abs{\nabla u}^2 \dm x + \beta \int_{J_u} \! (u^+)^2 + (u^-)^2 \dm \HH^{n-1} + \gamma^2 \LL^n(\set{u > 0}).
\end{equation}
It is shown in~\cite[Theorem 4.2, Corollary 3.3]{CK} that at least one minimizer exists and that there exists $\delta > 0$ (depending only on $\Omega$, $\beta$, $\gamma$) such that $\Omega \subset \subset B(0,\delta^{-1})$ and all minimizers have a compact support in $B(0,\delta^{-1})$.
So we can work without loss of generality with the competitors that have a compact support in $B(0,\delta^{-1})$.

Now, we state the notion of calibration associated to this problem.
The existence of such a vector field implies that $u$ is a minimizer of $E(v)$ among $v \in SBV(\R^n)$ such that $v = 1$ on $\Omega$ and $0 \leq v \leq 1$ in $\R^n$ (see Appendix~\ref{appendix_calibration}).
We define $\phi$ on $(\R^n \setminus \Omega) \times [0,1]$ but it would be enough for $\phi$ to be defined on $\left(B(0,\delta^{-1}) \setminus \Omega\right) \times [0,1]$.

\begin{defi}\label{defi_isolation_calibration}
    Let $\Omega$ be a bounded open set of $\R^n$ with Lipschitz boundary, let $u \in SBV(\R^n)$ be such that $u = 1$ on $\Omega$, $0 \leq u \leq 1$ in $\R^n$ and $u$ has a compact support.
    A calibration for $u$ is a Borel map
    \begin{equation}
        \phi = (\phi^x,\phi^t) \colon (\R^n \setminus \Omega) \times \R \to \R^n \times \R
    \end{equation}
    which is bounded and approximately regular on $(\R^n \setminus \Omega) \times [0,1]$, divergence-free on $(\R^n \setminus \overline{\Omega}) \times ]0,1[$ in the sense of distribution and such that
    \begin{enumerate}
        \item[(a)] $\phi^t(x,t) \geq \tfrac{1}{4} \abs{\phi^x(x,t)}^2 - \gamma^2 \mathbf{1}_{\set{t > 0}}(t)$ for $\LL^n$-a.e.\ $x \in \R^n \setminus \overline{\Omega}$ and every $t \in [0,1]$;
        \item[(b)] $\abs{\int_r^s \! \phi^x(x,t) \dm t} \leq \beta (r^2 + s^2)$ for $\HH^{n-1}$-a.e.\ $x \in \R^n \setminus \Omega$ and every $r,s \in [0,1]$;
        \item[(a')] $\phi^x(x,u) = 2 \nabla u$ and $\phi^t(x,u) = \abs{\nabla u}^2 - \gamma^2 \mathbf{1}_{\set{t > 0}}(u)$ for $\LL^n$-a.e.\ $x \in \R^n \setminus \overline{\Omega}$;
        \item[(b')] $\int_{u^-}^{u^+} \! \phi^x(x,t) \dm t = \beta \left[(u^-)^2 + (u^+)^2\right]\nu_u$ for $\HH^{n-1}$-a.e.\ $x \in J_u$.
    \end{enumerate}
\end{defi}

We are interested in finding minimality criteria in the case $\Omega = B(0,1)$ or $\Omega$ convex.
Similar questions were studied by De Pauw and Smets for the Mumford-Shah functional (\cite{DPS}), without calibrations.

\subsubsection{Informal computations with \texorpdfstring{$\Omega = B(0,1)$}{?© = B(0,1)}}\label{section_informal}

We present informal computations in the case $\Omega = B_1$ (the open unit ball centred at the origin).
We introduce for $r > 0$,
\begin{equation}\label{eq_gamma}
    \Gamma(r) =
    \begin{cases}
        r - 1                       & \ \text{if $n = 1$} \\
        \ln(r)                      & \ \text{if $n = 2$} \\
        \frac{1}{n-2} (1 - r^{2-n}) & \ \text{if $n \geq 3$}.
    \end{cases}
\end{equation}
Given $x \in \mathbf{R}^n$, we write $r$ for $\abs{x}$.
Thus, $x \mapsto \Gamma(r)$ is an harmonic positive function on $\R^n \setminus \overline{B_1}$ which is $0$ on $\partial B_1$.

We assume without proof that the only relevant competitors $u$ are of the following form.
Either $u$ is the indicator function of $B_1$.
Either there exists $R > 1$ and $0 < \delta < 1$ such that $u$ is radial continuous on $\overline{B}_R$, $u = 1$ on $\overline{B_1}$, $u = \delta$ on $\partial B_R$ and $u = 0$ on $\R^n \setminus \overline{B}_R$.
For a fixed $R > 1$, the first Euler-Lagrange equation says that $u$ should be harmonic in $B_R \setminus \overline{B_1}$ and that $\delta = \delta(R)$ should be determined by the Robin boundary condition
\begin{equation}
    -\partial_r u = \beta u \ \text{on $\partial B_R$}.
\end{equation}
We find
\begin{equation}\label{eq_delta}
    \delta(R) = \frac{1}{1 + \beta R^{n-1} \Gamma(R)}
\end{equation}
and
\begin{equation}\label{eq_u}
    u(x) =
    \begin{cases}
        1                                     & \ \text{for $\abs{x} \leq 1$}        \\
        1 - \beta \delta(R) R^{n-1} \Gamma(x) & \ \text{for $1 \leq \abs{x} \leq R$} \\
        0                                     & \ \text{for $\abs{x} > R$}.
    \end{cases}
\end{equation}
We plot an example of $\delta$ and $u$ on Figure \ref{fig_u}.
\begin{figure}[ht]
    \begin{center}
        \includegraphics[width = \linewidth]{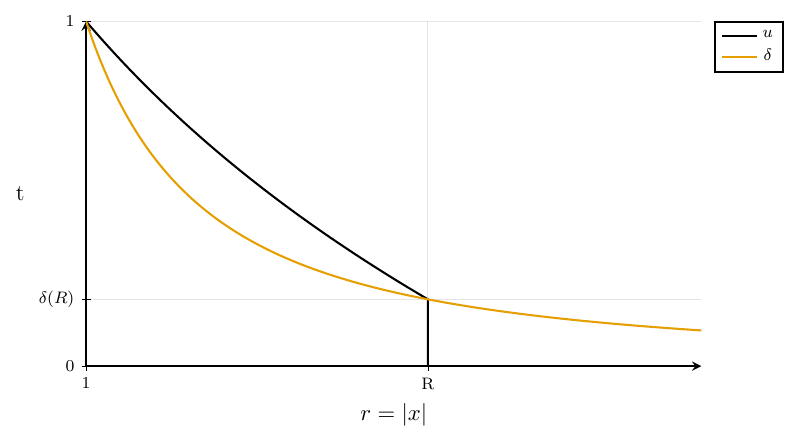}
        \caption{$n = 2$, $\beta = 3$ and $R = 2$}
        \label{fig_u}
    \end{center}
\end{figure}

An integration by parts combined with the Dirichlet condition $u = 1$ on $\partial B_1$ and the Robin condition $-\partial_r u = \beta u$ on $\partial B_R$ shows that
\begin{equation}
    \int_{B_R \setminus B_1} \! \abs{\nabla u}^2 \dm x + \beta \int_{\partial B_R} \! u^2 \dm \HH^{n-1} = \beta \int_{\partial B_R} \! u \dm \HH^{n-1}.
\end{equation}
We conclude that the energy of $u$ is
\begin{equation}
    E = n \omega_n \beta R^{n-1} \delta(R) + \omega_n \gamma^2 R^n.
\end{equation}
Now, we consider $E$ as functions of $R \in [1,+\infty[$ that we try to optimize.
We observe that $n \omega_n R^{n-1} \beta \delta(R)$ is the flux of the vector field $x \mapsto \beta \delta(r) e_r$ through $\partial B_R$.
We compute
\begin{align}
    \mathrm{div}(\delta(r) e_r) & = \delta'(r) + \frac{(n - 1)}{r} \delta(r) \\
                                & = - \left(\beta - \frac{(n - 1)}{r}\right) \delta(r)^2
\end{align}
so
\begin{equation}
    E(R) = E(1) + \int_{B_R \setminus B_1} \! \left[\gamma^2 - \left(\beta^2 - \frac{(n - 1) \beta}{r}\right) \delta(r)^2\right] \dm x,
\end{equation}
where $r$ means $\abs{x}$.
This shows in particular that
\begin{equation}\label{eq_Eprime}
    E'(R) = n \omega_n R^{n-1} \left[\gamma^2 - \left(\beta^2 - \frac{(n - 1) \beta}{R}\right) \delta(R)^2\right].
\end{equation}
The critical radii $R > 1$ are characterised by the equations
\begin{equation}
    \left(\beta^2 - \frac{\beta (n - 1)}{R}\right) \delta(R)^2 = \gamma^2.
\end{equation}
\begin{rmk}
    As expected, this coincides with the second Euler-Lagrange equation
    \begin{equation}
        \abs{\nabla u}^2 + \gamma^2 + \beta u^2 (H - 2 \beta) = 0
    \end{equation}
    where $H = (n - 1) R^{-1}$ is the mean scalar curvature of $\partial B_R$ with respect to $-e_r$ (see~\cite[Definition 7.32]{Ambrosio}, not to be confused with the arithmetic mean of the principal curvatures which is equal to $R^{-1}$).
    A proof of this formula for $C^{1,\alpha}$ surfaces is presented in~\cite[Theorem 15.1]{K}.
\end{rmk}
Depending on the parameters $\beta$, $\gamma$, the function $R \mapsto E(R)$ may not be convex and the condition $E'(R) = 0$ may not suffice to characterize minimizers.
See Figure~\ref{fig_E}.

\begin{figure}[ht]
    \begin{center}
        \includegraphics[width = 0.8\textwidth]{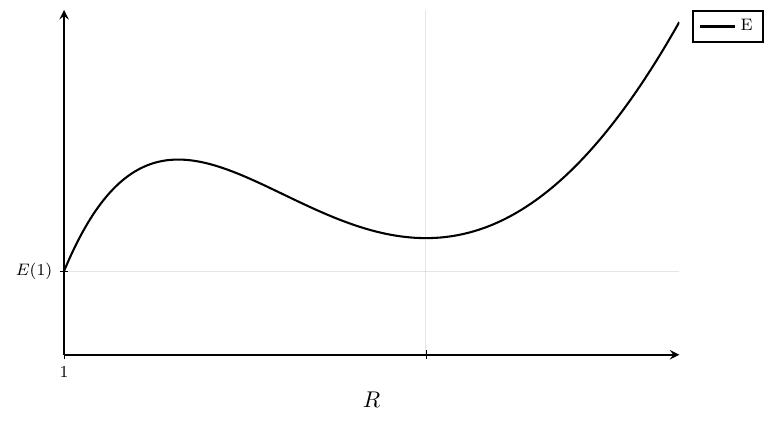}
        \caption{$n = 2$, $\beta = 1$ and $\gamma = 0.34$}
        \label{fig_E}
    \end{center}
\end{figure}

\subsubsection{Three sufficient conditions of minimality}

An sufficient condition for $\mathbf{1}_{B(0,1)}$ to be a minimizer is that the function $r \mapsto E(r)$ is non-decreasing on $[1,+\infty[$.
This is equivalent to the fact that for all $r \geq 1$,
\begin{equation}\label{eq_E_increasing}
    \left(\beta^2 - \frac{(n - 1) \beta}{r}\right) \delta(r)^2 \leq \gamma^2.
\end{equation}
In particular, it suffices that $\beta \leq \gamma$.
In the two next theorems, we generalize the sufficient conditions $\beta \leq \gamma$ and (\ref{eq_E_increasing}) to other domains.

\begin{thm}\label{thm_indicator1}
    Let $\Omega$ be a $C^1$ bounded open set of $\R^n$ and assume that its outward unit normal vector field has a continuous extension $\nu \colon \R^n \setminus \Omega \to \R^n$ such that $\abs{\nu} \leq 1$ and which is divergence-free on $\R^n \setminus \overline{\Omega}$ in the sense of distribution.
    If $\beta \leq \gamma$, then $\mathbf{1}_\Omega$ has a calibration.
\end{thm}

The proof of Theorem~\ref{thm_indicator1} is postponed in Section~\ref{section_idea1}.

\begin{rmk}
    The existence of such an extension holds true if $\Omega$ is a $C^2$ bounded open convex in $\R^2$ (\cite[Proposition 15]{C4}).
\end{rmk}

\begin{rmk}
    Such an extension calibrates $\Omega$ as an \emph{outward minimizing set}. This means that for all bounded set of finite perimeter $E \subset \R^n$ containing $\Omega$, we have $P(\Omega) \leq P(E)$.
Let us justify this claim. Let $\partial^* E$ denote the reduced boundary of $E$ and let $\nu_E$ denote the measure-theoretic outward normal vector field of $E$ (defined $\HH^{n-1}$ a.e.\ on $\partial^* E$).
We work in the ambient open set $X = \R^n \setminus \overline{\Omega}$ and we consider the set $F := E \setminus \overline{\Omega}$. In $X$, we observe that $F$ is of finite perimeter whose reduced boundary $\partial^* F$ coincides $\HH^{n-1}$-a.e. with $\partial^* E \setminus \partial \Omega$ and whose outward normal vector field $\nu_F$ coincides $\HH^{n-1}$-a.e. on $\partial^* F$ with $\nu_E$.
It is also easy to see that the trace of $u := \mathbf{1}_F$ on $\partial \Omega$ coincides $\HH^{n-1}$-a.e. with $\mathbf{1}_{\partial \Omega \setminus \partial^* E}$.

We apply the divergence theorem (\cite[Lemma 2.4]{Alberti}) in the domain $X$ and with the function $u = \mathbf{1}_{F}$. For all bounded approximately regular vector field $\phi$ on $\R^n \setminus \Omega$ with $\mathrm{div}(\phi) \in L^\infty(\R^n \setminus \overline{\Omega})$, we have
    \begin{equation}
        \int_{E \setminus \overline{\Omega}} \! \mathrm{div} (\phi) \dm x = \int_{\partial^* E \setminus \Omega} \! \phi \cdot \nu_E \dm \HH^{n-1} - \int_{\partial \Omega \setminus \partial^* E} \! \phi \cdot \nu \dm \HH^{n-1}.
    \end{equation}
    We pick $\phi = \nu$ and we get
    \begin{equation}
        \int_{\partial^* E \setminus \partial \Omega} \! (\nu \cdot \nu_E) \dm \HH^{n-1} = P(\Omega)
    \end{equation}
    but since $\abs{\nu} \leq 1$, we deduce that $P(\Omega) \leq P(E)$.

    Reciprocally, if $\Omega$ is outward miminizing, its outward unit normal vector field admits at least a \emph{weak extension} $\nu \in L^\infty(\R^n \setminus \overline{\Omega};\R^n)$ such that $\abs{\nu} \leq 1$ and which is divergence-free on $\R^n \setminus \overline{\Omega}$ in the sense of distribution.
    Here, \emph{weak extension} means that the normal component of $\nu$ on $\partial \Omega$ is defined in a functional sense (this notion is due to Anzellotti,~\cite{Anzellotti}).
    For the interested reader, we provide a proof of existence of $\nu$ communicated by A. Chambolle in Appendix~\ref{appendix_extension}.

    This weak extension $\nu$ may not be approximately regular. If we use it to build the calibration $\phi$ in Section~\ref{section_idea1}, $\phi$ may not be approximately regular either and its normal component may not be well-defined in the usual divergence theorem (\cite[Lemma 2.4]{Alberti}).
    However, for vector fields with $L^\infty$ distributional divergence (and in particular for divergence-free vectorfields), it is always possible to define the trace of the normal component in the sense of Anzellotti~\cite{Anzellotti} and the divergence theorem also holds true for this definition.
    The notion of calibration may thus be adapted to non approximately regular vector fields.
    It would be more difficult to check that a vector field implies minimality because we would no longer have pointwise requirements such as (a), (b), (a'), (b') in Definition~\ref{defi_isolation_calibration}.
    But the calibration that we present in Section~\ref{section_idea1} is quite simple and the competitor is just an indicator function.

    This hints that the condition $\beta \leq \gamma$ should be sufficient for the minimality of $\mathbf{1}_\Omega$, provided that $\Omega$ is a bounded open set of Lipschitz boundary which is outward minimizing.
    The proof should be a matter of rewriting things in the formalism of Anzellotti but we do not try in this paper which is quite long already.
\end{rmk}

Before stating the second theorem, we generalize the previous function $\Gamma$ to other domains $\Omega$.
The proof of the next lemma is postponed in Appendix~\ref{appendix_gamma} because this is not a new result.

\begin{lem}\label{lem_potential}
    Let $\Omega$ be a non-empty $C^2$ star-shaped bounded open set of $\R^n$.
    There exists a continuous function $\Gamma \colon \R^n \setminus \Omega \to \R$ such that
    \begin{enumerate}[label = (\roman*)]
        \item $\Gamma$ is harmonic positive on $\R^n \setminus \overline{\Omega}$;
        \item $\Gamma = 0$ on $\partial \Omega$,
        \item For all $x \in \R^n \setminus \overline{\Omega}$, $\nabla \Gamma(x) \ne 0$
        \item $\Gamma$ has a $C^1$ extension up to the boundary and for all $x \in \partial \Omega$, there exists $t > 0$ such that $\nabla \Gamma(x) = t \nu(x)$, where $\nu(x)$ is the outward unit normal vector of $\Omega$ at $x$.
    \end{enumerate}
\end{lem}

\begin{thm}\label{thm_indicator2}
    Let $\Omega$ be a non-empty $C^2$ star-shaped bounded open set of $\R^n$ and let $\Gamma$ be a function as in Lemma~\ref{lem_potential} (modulo a positive multiplicative constant).
    We define for $x \in \R^n \setminus \overline{\Omega}$,
    \begin{equation}
        \nu(x) = \frac{\nabla \Gamma}{\abs{\nabla \Gamma}} \qquad \text{and} \qquad \delta(x) = \frac{1}{1 + \beta \abs{\nabla \Gamma}^{-1} \Gamma}
    \end{equation}
    and we assume that $\mathrm{div}(\nu)$ is bounded.
    If for all $x \in \R^n \setminus \overline{\Omega}$,
    \begin{equation}\label{eq_indicator2_condition}
        \left(\beta^2 - \beta \mathrm{div}(\nu)\right) \delta^2 \leq \gamma^2,
    \end{equation}
    then $\mathbf{1}_\Omega$ has a calibration.
\end{thm}

In the case $\Omega = B(0,1)$, we have $\nu = e_r$ and the condition (\ref{eq_indicator2_condition}) just amounts to (\ref{eq_E_increasing}).
The proof of Theorem~\ref{thm_indicator2} is postponed in Section~\ref{section_idea2}.

Finally, we come back to $\Omega = B(0,1)$ and the notations of Section~\ref{section_informal}.
In particular, we refer to (\ref{eq_gamma}) and (\ref{eq_delta}) for the definition of $\Gamma$ and $\delta$.
We are going to see that if $\beta \geq n - \tfrac{1}{2}$, the Euler-Lagrange equations characterize minimizers.
In view of the formula
\begin{equation}
    E'(r) = n \omega_n r^{n-1} \left[\gamma^2 - \left(\beta^2 - \frac{(n - 1) \beta}{r}\right) \delta(r)^2\right],
\end{equation}
it suffices to show that the function
\begin{equation}
    r \mapsto \left(\beta^2 - \frac{(n - 1) \beta }{r}\right) \delta(r)^2
\end{equation}
is decreasing on $[1,+\infty[$.
We write
\begin{equation}
    \left(\beta^2 - \frac{(n - 1) \beta }{r}\right) \delta(r) = \beta \left(\beta - \frac{n-1}{r}\right) \frac{(r^{n-1}\delta(r))^2}{r^{2n - 2}}.
\end{equation}
If $\beta \geq n - 1$, the function $r \mapsto r^{n-1} \delta(r)$ is decreasing on $[1,+\infty[$ because
\begin{equation}
    (r^{n-1} \delta)' = - r^{n-1} \left(\beta^2 - \frac{\beta (n - 1)}{r}\right) \delta(r)^2.
\end{equation}
If $\beta \geq n - \tfrac{1}{2}$, the function $r \mapsto \left(\beta - \frac{n-1}{r}\right) \frac{1}{r^{2n - 2}}$ is decreasing on $[1,+\infty[$ because
\begin{equation}
    \left[\left(\beta - \frac{n-1}{r}\right) \frac{1}{r^{2n - 2}}\right]' = -2 \frac{(n - 1)}{r^{2n - 1}} \left(\beta - \frac{n - \tfrac{1}{2}}{r}\right)
\end{equation}
This proves our claim.
In the next theorem, we build a calibration corresponding to this criteria.



\begin{thm}\label{thm_u}
    Let $\Omega$ be the open ball $B(0,1)$.
    We assume that $\beta \geq n - \tfrac{1}{2}$ and that there exists $R \geq 1$ such that
    \begin{equation}
        \left(\beta^2 - \frac{\beta (n - 1)}{R}\right)\delta(R)^2 = \gamma^2,
    \end{equation}
    Then the function
    \begin{equation}\label{eq_u2}
        u(x) =
        \begin{cases}
            1                                     & \ \text{for $\abs{x} \leq 1$}        \\
            1 - \beta \delta(R) R^{n-1} \Gamma(x) & \ \text{for $1 \leq \abs{x} \leq R$} \\
            0                                     & \ \text{for $\abs{x} > R$},
        \end{cases}
    \end{equation}
    has a calibration.
\end{thm}

The proof of Theorem~\ref{thm_u} is postponed in Section~\ref{section_idea3}.

\subsection{Proof of Theorem~\ref{thm_indicator1}}\label{section_idea1}

Let $\Omega$ be a $C^1$ bounded open set of $\R^n$ (we will add more assumptions as we advance in the construction).
We want to build a calibration for $\mathbf{1}_\Omega$ and we search for a continuous function $\phi \colon (\R^n \setminus \Omega) \times [0,1] \to \R^n \times \R$ which is divergence-free in $\R^n \setminus \overline{\Omega}$ and such that
\begin{enumerate}
    \item[(a)] $\phi^t(x,t) \geq \tfrac{1}{4} \abs{\phi^x(x,t)}^2 - \gamma^2 \mathbf{1}_{\set{t > 0}}(t)$ for all $x \in \R^n \setminus \overline{\Omega}$ and $t \in [0,1]$;
    \item[(b)] $\abs{\int_r^s \! \phi^x(x,t) \dm t} \leq \beta (r^2 + s^2)$ for all $x \in \R^n \setminus \Omega$ and $r,s \in [0,1]$;
    \item[(a')] $\phi^x(x,0) = 0$ and $\phi^t(x,0) = 0$ for all $x \in \R^n \setminus \overline{\Omega}$;
    \item[(b')] $\int_{0}^{1} \! \phi^x(x,t) \dm t = -\beta \nu(x)$ for all $x \in \partial \Omega$, where $\nu$ is the outward unit normal vector field of $\Omega$.
\end{enumerate}
We look for $\phi^x$ first and then we will derive $\phi^t$.
The starting point is to define $\phi^x$ on the jump part of the complete graph of $\mathbf{1}_\Omega$ in such way that $t \mapsto \phi^x(x,t)$ is linear;
we define for $x \in \partial \Omega$ and $0 \leq t \leq 1$,
\begin{equation}
    \phi^x(x,t) = - 2 \beta t \nu(x).
\end{equation}
Let us assume that $\nu$ has a continuous extension on $\R^n \setminus \Omega$.
We can extend the previous formula and define for all $x \in \R^n \setminus \Omega$ and $0 \leq t \leq 1$,
\begin{equation}
    \phi^x(x,t) = -2 \beta t \nu(x).
\end{equation}
We have then for $x \in \R^n \setminus \Omega$ and for $r, s \in [0,1]$,
\begin{equation}
    \abs{\int_r^s \! \phi^x \dm t} \leq \beta (r^2 + s^2) \abs{\nu(x)}
\end{equation}
because $\abs{\phi^x(x,t)} \leq 2 \beta t \abs{\nu(x)}$.
The requirement (b) leads us to assume that $\abs{\nu} \leq 1$ in $\R^n \setminus \Omega$.
Next, we assume that $\nu$ has a $L^\infty$ distributional divergence in $\R^n \setminus \overline{\Omega}$.
We have then $\mathrm{div}(\phi) = 2 \beta t \mathrm{div}(\nu) + \partial_t \phi^t$ in the sense of distribution
so the conditions $\mathrm{div}(\phi) = 0$ and $\phi^t(x,0) = 0$ impose
\begin{equation}
    \phi^t(x,t) = \beta t^2 \mathrm{div}(\nu).
\end{equation}
The last requirement $\phi^t \geq \frac{1}{4} \abs{\phi^x}^2 - \gamma^2 \mathbf{1}_{\set{t > 0}}(t)$ amounts to
\begin{equation}
    \gamma^2 \geq \beta^2 \abs{\nu}^2 - \beta \mathrm{div}(\nu).
\end{equation}
In view of $\abs{\nu} \leq 1$, the simplest choice is to assume $\mathrm{div}(\nu) = 0$ and $\gamma \geq \beta$.

\subsection{Proof of Theorem~\ref{thm_indicator2}}\label{section_idea2}

Let $\Omega$ be a $C^1$ bounded open set of $\R^n$ (we will add more assumptions as we advance in the construction).
We are more careful than in the previous section and we define $\phi$ in two pieces.
However, we still arrange $\phi$ to be be globally continuous.

The principle is the same as before.
We let $\nu$ denote the outward unit normal vector field of $\Omega$.
The starting point is to define $\phi^x(x,t)$ for $x \in \partial \Omega$ and $0 \leq t \leq 1$ by $\phi^x(x,t) = -2 \beta t \nu(x)$.
We assume that $\nu$ has a continuous extension on $\R^n \setminus \Omega$ such that $\abs{\nu} \leq 1$ and $\nu$ is $C^1$ on $\R^n \setminus \overline{\Omega}$.
We also consider a continuous function $\delta \colon \R^n \setminus \Omega \to [0,1]$ which is $C^1$ on $\R^n \setminus \overline{\Omega}$ and such that $\delta = 1$ on $\partial \Omega$ and $0 < \delta < 1$ on $\R^n \setminus \overline{\Omega}$.
\begin{figure}[ht]
    \begin{center}
        \includegraphics[width = \linewidth]{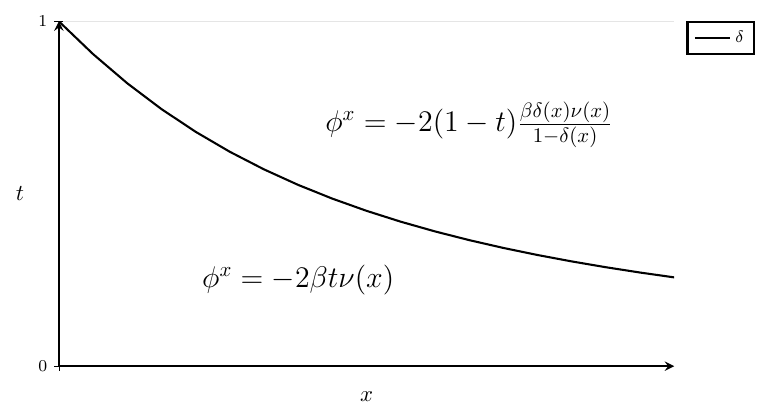}
    \end{center}
\end{figure}
We define for $x \in \R^n \setminus \overline{\Omega}$ and $0 \leq t \leq \delta(x)$,
\begin{align}
    \phi^x & = -2 \beta t \nu \\
    \phi^t & = \beta t^2 \mathrm{div}(\nu)
\end{align}
and for $\delta(x) < t \leq 1$,
\begin{align}
    \phi^x & = -2 (1 - t) \frac{\beta \delta \nu}{1 - \delta} \\
    \phi^t & = -(1 - t)^2 \mathrm{div}\left(\frac{\beta \delta \nu}{1 - \delta}\right) - C(x)
\end{align}
where $C(x)$ will be chosen so that $\phi$ is continuous on $(\R^n \setminus \overline{\Omega}) \times [0,1]$.
Observe that this is already the case for $\phi^x$.
We find
\begin{equation}
    C(x) = - (1 - \delta)^2 \mathrm{div}\left(\frac{\beta \delta \nu}{1 - \delta}\right) - \beta \delta^2 \mathrm{div}(\nu).
\end{equation}
However we compute
\begin{equation}\label{eq_div_computation}
    \mathrm{div}\left(\frac{\delta \nu}{1 - \delta}\right) = \frac{\mathrm{div}(\delta \nu) - \delta^2 \mathrm{div}(\nu)}{(1 - \delta)^2}
\end{equation}
so this simplifies to $C(x) = -\beta \mathrm{div}(\delta \nu)$.
With regard to approximate regularity, the definition of $\phi^t$ on $\partial \Omega \times [0,1]$ does not matter.
Indeed, let $M$ be an hypersurface of $\R^n \times \R$.
Then for $\HH^n$-a.e.\ $(x_0,t_0) \in M \cap (\partial \Omega \times [0,1])$, the vector $n_0 = (\nu(x_0),0)$ is a normal vector to $M$ at $x$ and $\phi \cdot n_0 = \phi^x$.
In order for $\phi$ to be bounded, it suffices that $\mathrm{div}(\nu)$ and $\mathrm{div}(\delta \nu)$ are bounded.
Finally, the requirement $\phi^t \geq \frac{1}{4} \abs{\phi^x}^2 - \gamma^2 \mathbf{1}_{\set{t > 0}}(t)$ amounts to
\begin{subequations}\label{eq_idea2_conditions}
    \begin{align}
        \left(\beta^2 \abs{\nu}^2 - \beta \mathrm{div}(\nu)\right) \delta^2 & \leq \gamma^2\label{eq_sigma1} \\
        - \beta \mathrm{div}(\delta \nu)                                    & \leq \gamma^2\label{eq_sigma2}.
    \end{align}
\end{subequations}
It is tempting to choose $\delta$ and $\nu$ in such a way that
\begin{equation}\label{eq_delta_nu_equality}
    \left(\beta^2 \abs{\nu}^2 - \beta \mathrm{div}(\nu)\right) \delta^2 = - \beta \mathrm{div}(\delta \nu).
\end{equation}
In that case, $\mathrm{div}(\delta \nu)$ is bounded provided that $\mathrm{div}(\nu)$ is bounded.
According to (\ref{eq_div_computation}), the equality (\ref{eq_delta_nu_equality}) is equivalent to
\begin{equation}
    \left(\frac{\beta \abs{\delta \nu}}{1 - \delta}\right)^2 = -\mathrm{div}\left(\frac{\beta \delta \nu}{1 - \delta}\right)
\end{equation}
A natural solution is to assume that $\Omega$ is $C^2$ star-shaped, to consider the function $\Gamma$ of Lemma~\ref{lem_potential} and to choose $\delta$, $\nu$ in such a way that
\begin{equation}
    \frac{\beta \delta \nu}{1 - \delta} = \frac{\nabla \Gamma}{\Gamma}.
\end{equation}
We suggest to define
\begin{equation}
    \nu(x) = \frac{\nabla \Gamma}{\abs{\nabla \Gamma}}
\end{equation}
and
\begin{equation}
    \delta(x) = \frac{1}{1 + \beta \abs{\nabla \Gamma}^{-1} \Gamma}.
\end{equation}
In conclusion, the conditions (\ref{eq_idea2_conditions}) simplify to
\begin{equation}
    \left(\beta^2 - \beta \mathrm{div}(\nu)\right) \delta^2 \leq \gamma^2.
\end{equation}

\subsection{Proof of Theorem~\ref{thm_u}}\label{section_idea3}

Let $\Omega$ be the open ball $B(0,1)$.
We assume that $\beta \geq n - \tfrac{1}{2}$ and that there exists $R \geq 1$ such that
\begin{equation}\label{eq_u_condition}
    \left(\beta^2 - \frac{\beta (n - 1)}{r}\right)\delta(r)^2 = \gamma^2,
\end{equation}
We have seen just before the statement of~\ref{thm_u} that when $\beta \geq n - \tfrac{1}{2}$, the function
\begin{equation}
    r \mapsto \left(\beta^2 - \frac{(n - 1) \beta}{r}\right) \delta(r)^2
\end{equation}
is decreasing on $[1,+\infty[$.
Therefore, (\ref{eq_u_condition}) implies that for all $r \geq R$,
\begin{equation}
    \left(\beta^2 - \frac{\beta (n - 1)}{r}\right)\delta(r)^2 \leq \gamma^2,
\end{equation}
The case $R = 1$ has already been dealt with in Theorem~\ref{thm_indicator2} so we can assume $R > 1$.

We recall the notations.
Given $x \in \mathbf{R}^n$, we write $r$ for $\abs{x}$.
We define for $r \geq 1$,
\begin{equation}\label{eq_gamma2}
    \Gamma(r) =
    \begin{cases}
        r - 1                       & \ \text{if $n = 1$} \\
        \ln(r)                      & \ \text{if $n = 2$} \\
        \frac{1}{n-2} (1 - r^{2-n}) & \ \text{if $n \geq 3$}.
    \end{cases}
\end{equation}
and
\begin{equation}
    \delta(r) = \frac{1}{1 + \beta r^{n-1} \Gamma(r)}.
\end{equation}
We define for $x \in \R^n$,
\begin{equation}
    u(x) =
    \begin{cases}
        1                                     & \ \text{for $\abs{x} \leq 1$}        \\
        1 - \beta \delta(R) R^{n-1} \Gamma(r) & \ \text{for $1 \leq \abs{x} \leq R$} \\
        0                                     & \ \text{for $\abs{x} > R$}.
    \end{cases}
\end{equation}
For $r \geq 0$, we write $u(r)$ for the value of $u$ on $\partial B_r$.
Thus, we consider $u$ as a function of the real variable $r \in [0,+\infty[$.
Now, we list a few useful formulas.
For $1 < \abs{x} < R$, we have
\begin{equation}\label{eq_nabla_u}
    \nabla u = - \beta \delta(R) \left(\tfrac{R}{r}\right)^{n-1} e_r
\end{equation}
and
\begin{equation}
    \frac{\nabla u}{1 - u} = - \frac{\beta \delta}{1 - \delta} e_r = - \frac{1}{r^{n-1} \Gamma(r)} e_r.
\end{equation}
With a slight abuse of notations, we consider that $\nabla u$ is defined on $1 \leq \abs{x} \leq R$ by (\ref{eq_nabla_u}).
We observe that
\begin{equation}
    \mathrm{div}\left(\frac{\nabla u}{1 - u}\right) = \left(\frac{\abs{\nabla u}}{1 - u}\right)^2.
\end{equation}
or equivalently
\begin{equation}
    -\mathrm{div}\left(\frac{\beta \delta e_r}{1 - \delta}\right) = \left(\frac{\beta \delta}{1 - \delta}\right)^2.\label{eq_div_delta2}
\end{equation}
According to (\ref{eq_div_computation}), the line (\ref{eq_div_delta2}) is also equivalent to
\begin{equation}
    -\beta\mathrm{div}(\delta e_r) = \left(\beta^2 - \frac{\beta (n - 1)}{r}\right)\delta(r)^2.
\end{equation}

We are going to define the calibration.
Although we define $\phi^x$ and $\phi^t$ in parallel, the relevant part is really $\phi^x$.
The function $\phi^t$ is derived as usual.
We consider a continuous function $\rho \colon [1,R] \to \R$ which is $C^1$ on $[1,R]$ and such that $\delta(R) \leq \rho \leq u$.
\begin{figure}[ht]
    \begin{center}
        \includegraphics[width = \linewidth]{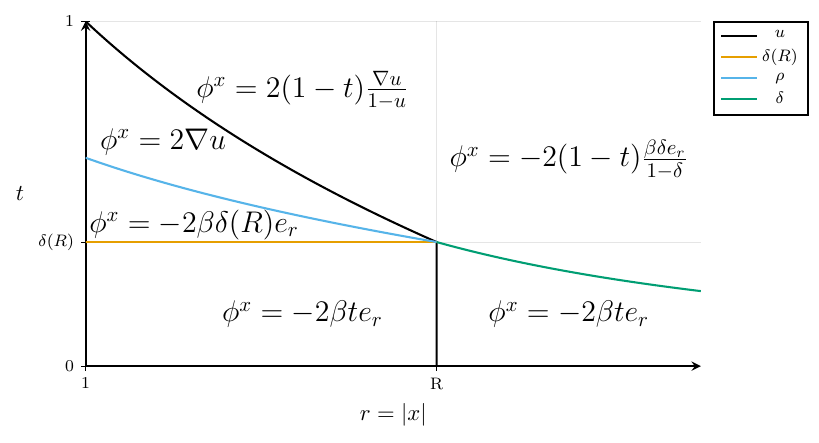}
    \end{center}
\end{figure}

We fix $x$ such that $1 \leq \abs{x} \leq R$.
We define for $0 \leq t \leq \delta(R)$,
\begin{align}
    \phi^x & = -2 \beta t e_r \\
    \phi^t & = \frac{(n - 1)\beta t^2}{r},
\end{align}
for $\delta(R) \leq t < \rho(r)$
\begin{align}
    \phi^x & = -2 \beta \delta(R) e_r \\
    \phi^t & = \frac{2 (n - 1) \beta \delta(R) t}{r} - \frac{(n - 1) \beta \delta(R)^2}{r}
\end{align}
for $\rho(r) < t \leq u(r)$
\begin{align}
    \phi^x(x,t) & = 2 \nabla u                                             \\
                & = -2 \beta \delta(R) \left(\tfrac{R}{r}\right)^{n-1} e_r \\
    \phi^t(x,t) & = \abs{\nabla u}^2 - \gamma^2                            \\
                & = \beta^2 \delta(R)^2 \left(\tfrac{R}{r}\right)^{2n - 2} - \gamma^2
\end{align}
and for $u(r) \leq t \leq 1$,
\begin{align}
    \phi^x(x,t) & = 2 (1 - t) \frac{\nabla u}{1 - u}                                 \\
                & = -2 (1 - t) \frac{\beta \delta e_r}{1 - \delta}                   \\
    \phi^t(x,t) & = (1 - t)^2 \left(\frac{\abs{\nabla u}}{1 - u}\right)^2 - \gamma^2 \\
                & = (1 - t)^2 \left(\frac{\beta \delta}{1 - \delta}\right)^2 - \gamma^2
\end{align}
Next, we fix $x$ such that $\abs{x} \geq R$ and we use the same formula as in Section~\ref{section_idea2}.
We define for $t \leq \delta(r)$,
\begin{align}
    \phi^x & = -2 \beta t e_r \\
    \phi^t & = \frac{(n - 1) \beta t^2}{r}
\end{align}
and for $t \geq \delta(r)$,
\begin{align}
    \phi^x(x,t) & = -2 (1 - t) \frac{\beta \delta e_r}{1 - \delta} \\
    \phi^t(x,t) & = (1 - t)^2 \left(\frac{\beta \delta}{1 - \delta}\right)^2 - \left(\beta^2 - \frac{\beta (n - 1)}{r}\right) \delta^2
\end{align}
When $n = 1$, the function $\rho$ is given by the constant $\delta(R)$.
When $n \geq 2$, it is given by the following complicated formula that the reader can ignore for the moment,
\begin{multline}
    \rho(r)  = \frac{\delta(R)}{2} + \frac{\beta \delta(R) r}{2} \left(\frac{\left(\tfrac{R}{r}\right)^{2n - 2} \Gamma\left(\tfrac{R}{r}\right)}{\left(\tfrac{R}{r}\right)^{n-1} - 1}\right) \\ - \frac{\delta(R) r}{2n}\left(\beta - \frac{n-1}{R}\right) \left(\frac{\left(\tfrac{R}{r}\right)^n - 1}{\left(\tfrac{R}{r}\right)^{n-1} - 1}\right).
\end{multline}
Regardless of $\rho$, many properties can be checked.
The vector field $\phi$ is bounded.
It is continuous outside the graph
\begin{equation}
    \set{(x,t) | t = \rho(r),\ 1 \leq r < R}.
\end{equation}
We point out that it is continuous through $\abs{x} = R$ because
\begin{equation}
    \gamma^2 = \left(\beta^2 - \frac{\beta (n - 1)}{R}\right) \delta(R)^2
\end{equation}
and because for $\abs{x} = R$, $\nabla u(x) = -\beta \delta(R) e_r$.
The function $\phi$ is divergence free in the interior of each part.
In order for $\phi$ to be divergence-free in $\R^n \setminus \overline{\Omega}$ (in the sense of distributions), we have to choose $\rho$ in such a way that for all $1 < \abs{x} < R$
\begin{equation}
    \phi(x,\rho(r)^-)
    \cdot
    \begin{pmatrix}
        - \rho'(r) e_r \\ 1
    \end{pmatrix}
    =
    \phi(x,\rho(r)^+)
    \cdot
    \begin{pmatrix}
        - \rho'(r) e_r \\ 1
    \end{pmatrix}.
\end{equation}
This will imply that $\phi$ is approximately regular on $(\R^n \setminus \overline{\Omega}) \times [0,1]$ by~\cite[Remark 2.6]{Alberti}.
Next, we are going to deal with the the approximately regularity of $\phi$ on on $\partial \Omega \times [0,1]$ as in the previous section.
Let $M$ be an hypersurface of $\R^n \times \R$.
For $\HH^n$-a.e.\ $(x_0,t_0) \in M \cap (\partial \Omega \times [0,1])$, we have $t_0 \ne \rho(x_0)$ and the vector $n_0 = (\nu(x_0),0)$ is a normal vector to $M$ at $x$.
Thus $\phi \cdot n_0 = \phi^x$ and we conclude by continuity of $\phi^x$ in an neighborhood of $(x_0,t_0)$.

We check the values of $\phi(x,u(x))$.
It is clear from the construction that for $1 < \abs{x} < R$,
\begin{align}
    \phi^x(x,u) & = 2 \nabla u \\
    \phi^t(x,u) & = \abs{\nabla u}^2 - \gamma^2,
\end{align}
that for $\abs{x} > R$,
\begin{align}
    \phi^x(x,u) & = 0 = 2 \nabla u \\
    \phi^t(x,u) & = 0 = \abs{\nabla u}^2
\end{align}
and that for $\abs{x} = R$,
\begin{equation}
    \int_0^{\delta(R)} \! \phi^x(x,t) \dm t = \beta \delta(R)^2.
\end{equation}
The requirement $\phi^t \geq \tfrac{1}{4}\abs{\phi^x}^2 - \gamma^2 \mathbf{1}_{\set{t > 0}}(t)$ holds true because for every $r \geq R$,
\begin{equation}
    \gamma^2 \geq \left(\beta^2 - \frac{\beta (n - 1)}{r}\right) \delta(r)^2.
\end{equation}
It is left to compute $\rho$ and to check that $\abs{\int_s^t \! \phi^x} \leq \beta (s^2 + t^2)$.
From now on, the proof is quite computational and the reader is free to skip.

We recall that we are looking for a continuous function $\rho \colon [1,R] \to [0,1]$ which is $C^1$ on $[1,R[$ and such that $\delta(R) \leq \rho \leq u$ and such that for $1 < \abs{x} < R$,
\begin{equation}\label{eq_phi_computations}
    \phi(x,\rho(r)^-)
    \cdot
    \begin{pmatrix}
        - \rho'(r) e_r \\ 1
    \end{pmatrix}
    =
    \phi(x,\rho(r)^+)
    \cdot
    \begin{pmatrix}
        -\rho'(r) e_r \\ 1
    \end{pmatrix}.
\end{equation}
When $n = 1$, we can take $\rho$ equals to the constant $\delta(R)$.
In this case, one can check that $\phi$ is continuous on $\left(\R^n \setminus \Omega\right) \times [0,1]$ and $\abs{\phi^x} \leq 2 \beta t$ so (\ref{eq_phi_computations}) and $\abs{\int_r^s \! \phi^x} \leq \beta (r^2 + s^2)$ hold true.
We pass to the case $n \geq 2$.
It will be convenient to express (\ref{eq_phi_computations}) in divergence form.
We compute
\begin{align}
    \phi(x,\rho^-)
    \cdot
    \begin{pmatrix}
        - \rho' e_r \\ 1
    \end{pmatrix}
                     & = 2 \beta \delta(R) \rho' + \frac{2 (n - 1) \beta \delta(R) \rho}{r} - \frac{(n - 1) \beta \delta(R)^2}{r} \\
                     & = 2 \beta \delta(R) \mathrm{div}(\rho e_r) - \beta \delta(R)^2 \mathrm{div}(e_r).
\end{align}
Next, we compute
\begin{align}
    \phi(x,\rho^+)
    \cdot
    \begin{pmatrix}
        -\rho' e_r \\ 1
    \end{pmatrix}
                    & = 2 \beta \delta(R) \left(\tfrac{R}{r}\right)^{n-1} \rho' + \beta^2 \delta(R)^2 \left(\tfrac{R}{r}\right)^{2n - 2} - \gamma^2 \\
                    \begin{split}
                    & = 2 \beta \delta(R) R^{n-1} \mathrm{div}\left(r^{1 - n} \rho e_r\right)                                                      \\
                    & \hspace{30pt} + \beta^2 \delta(R)^2 R^{2n - 2} \mathrm{div}\left(r^{1 - n} \Gamma(r) e_r\right) - \gamma^2
                    \end{split}
\end{align}
and we observe that we can write
\begin{equation}
    \gamma^2 = \mathrm{div}\left(\frac{\gamma^2 r}{n} e_r + \frac{c}{r^{n-1}} e_r\right)
\end{equation}
where $c$ is any real constant.
A natural solution is to choose $\rho$ such that
\begin{multline}
    2 \beta \delta(R) \rho - \beta \delta(R)^2 \\ = 2 \beta \delta(R) \left(\tfrac{R}{r}\right)^{n-1} \rho + \beta^2 \delta(R)^2 R^{n-1} \left(\tfrac{R}{r}\right)^{n-1} \Gamma(r) - \frac{\gamma^2 r}{n} + \frac{c}{r^{n-1}}.
\end{multline}
We rewrite this
\begin{multline}
    2 \beta \delta(R) \left(\left(\tfrac{R}{r}\right)^{n-1} - 1\right)\rho = -\beta \delta(R)^2 - \beta^2 \delta(R)^2 R^{n-1} \left(\tfrac{R}{r}\right)^{n-1} \Gamma(r) \\ + \frac{\gamma^2 r}{n} - \frac{c}{r^{n-1}}.
\end{multline}
We choose $c$ in such a way that the left-hand side cancels at $r = R$ (otherwise $\rho$ would have a singularity at $r = R$).
This yields
\begin{multline}
    2 \beta \delta(R) \left(\left(\tfrac{R}{r}\right)^{n-1} - 1\right)\rho = \beta \delta(R)^2 \left(\left(\tfrac{R}{r}\right)^{n-1} - 1\right) \\ + \beta^2 \delta(R)^2 R^{n-1} \left(\tfrac{R}{r}\right)^{n-1} (\Gamma(R) - \Gamma(r)) - \frac{\gamma^2 r}{n} \left(\left(\tfrac{R}{r}\right)^n - 1\right).
\end{multline}
Remember that $\Gamma(R) - \Gamma(r) = r^{2 - n} \Gamma\left(\tfrac{R}{r}\right)$ so
\begin{multline}
    2 \beta \delta(R) \left(\left(\tfrac{R}{r}\right)^{n-1} - 1\right)\rho = \beta \delta(R)^2 \left(\left(\tfrac{R}{r}\right)^{n-1} - 1\right) \\ + \beta^2 \delta(R)^2 r \left(\tfrac{R}{r}\right)^{2n - 2} \Gamma\left(\tfrac{R}{r}\right) - \frac{\gamma^2 r}{n} \left(\left(\tfrac{R}{r}\right)^n - 1\right).
\end{multline}
We conclude that
\begin{equation}
    \rho(r) = \frac{\delta(R)}{2} + \frac{\beta \delta(R) r}{2} \left(\frac{\left(\tfrac{R}{r}\right)^{2n - 2} \Gamma\left(\tfrac{R}{r}\right)}{\left(\tfrac{R}{r}\right)^{n-1} - 1}\right) - \frac{\gamma^2 r}{2n \beta \delta(R)} \left(\frac{\left(\tfrac{R}{r}\right)^n - 1}{\left(\tfrac{R}{r}\right)^{n-1} - 1}\right).
\end{equation}
As $\gamma^2 = \left(\beta^2 - \frac{\beta (n - 1)}{R}\right) \delta(R)^2$, an alternative expression is
\begin{multline}\label{eq_rho}
    \rho(r) = \frac{\delta(R)}{2} + \frac{\beta \delta(R) r}{2} \left(\frac{\left(\tfrac{R}{r}\right)^{2n - 2} \Gamma\left(\tfrac{R}{r}\right)}{\left(\tfrac{R}{r}\right)^{n-1} - 1}\right) \\ - \frac{\delta(R) r}{2n}\left(\beta - \frac{n-1}{R}\right) \left(\frac{\left(\tfrac{R}{r}\right)^n - 1}{\left(\tfrac{R}{r}\right)^{n-1} - 1}\right).
\end{multline}
It is clear that $\rho$ is a continuous function of $r \in ]1,R[$ and we also have $\lim_{r \to R} \rho(r) = \delta(R)$ because
\begin{equation}
    \lim_{t \to 1^+} \frac{\Gamma(t)}{t^{n-1} - 1} = \frac{1}{n-1}
\end{equation}
and
\begin{equation}
    \lim_{t \to 1^+} \frac{t^n - 1}{t^{n-1} - 1} = \frac{n}{n-1}.
\end{equation}

We show that $\rho \geq \delta(R)$.
It suffices to show that $\rho$ is decreasing on $[1,R[$.
For $0 \leq r < R$, we write $t = \tfrac{R}{r}$ so that
\begin{multline}
    \rho(r) = \frac{\delta(R)}{2} + \frac{\beta \delta(R) R t^{n - 2}}{2} \left(\frac{t^{n-1}\Gamma(t)}{t^{n-1} - 1}\right) \\ - \frac{\delta(R)}{2n}\left(\beta - \frac{n-1}{R}\right) t^{-1} \left(\frac{t^n - 1}{t^{n-1} - 1}\right).
\end{multline}
According to Lemma~\ref{lem_gamma}, the function $t \mapsto \left(\frac{t^{n-1}\Gamma(t)}{t^{n-1} - 1}\right)$ is increasing on $]1,+\infty[$.
It is also easy to see that the function $t^{-1} \left(\frac{t^n - 1}{t^{n-1} - 1}\right)$ is decreasing on $]1,+\infty[$ because
\begin{equation}
    t^{-1} \left(\frac{t^n - 1}{t^{n-1} - 1}\right) = 1 + t^{-1} \frac{t - 1}{t^{n-1} - 1}
\end{equation}
and $t \mapsto t^{n-1}$ is convex.
We deduce that $r \mapsto \rho(r)$ is decreasing on $[1,R[$.

Next, we show that for all $1 \leq r < R$, we have $\rho(r) \leq u(r)$.
Remember that
\begin{equation}
    u   = \delta(R) \left[1 + \beta r \left(\tfrac{R}{r}\right)^{n-1} \Gamma\left(\tfrac{R}{r}\right)\right]
\end{equation}
so we have to show that for all $1 \leq r < R$,
\begin{multline}
    \beta r \left(\frac{\left(\tfrac{R}{r}\right)^{2n - 2} \Gamma\left(\tfrac{R}{r}\right)}{\left(\tfrac{R}{r}\right)^{n-1} - 1}\right) \\ - \frac{r}{n}\left(\beta - \frac{n-1}{R}\right) \left(\frac{\left(\tfrac{R}{r}\right)^n - 1}{\left(\tfrac{R}{r}\right)^{n-1} - 1}\right) \leq 1 + 2 \beta r \left(\tfrac{R}{r}\right)^{n-1} \Gamma\left(\tfrac{R}{r}\right)
\end{multline}
We rewrite this,
\begin{multline}
    \beta r \left(\frac{\left(\tfrac{R}{r}\right)^{n-1} \Gamma\left(\tfrac{R}{r}\right)}{\left(\tfrac{R}{r}\right)^{n-1} - 1}\right) \left(2 - \left(\tfrac{R}{r}\right)^{n-1}\right) \\ - \frac{\beta r}{n} \left(\frac{\left(\tfrac{R}{r}\right)^n - 1}{\left(\tfrac{R}{r}\right)^{n-1} - 1}\right) \leq 1 - \frac{n-1}{n}\frac{r}{R} \left(\frac{\left(\tfrac{R}{r}\right)^n - 1}{\left(\tfrac{R}{r}\right)^{n-1} - 1}\right).
\end{multline}
It suffices to check that for all $t > 1$,
\begin{align}
    \beta \left(\frac{t^{n-1} \Gamma(t)}{t^{n-1} - 1}\right) \left(2 - t^{n-1}\right) - \frac{\beta}{n} \left(\frac{t^n - 1}{t^{n-1} - 1}\right) & \leq 0 \\
    1 -  \frac{n-1}{n} t^{-1} \left(\frac{t^n - 1}{t^{n-1} - 1}\right)                                                                            & \geq 0.
\end{align}
The first point can be simplified as
\begin{equation}
    \left(\frac{t^{n-1} \Gamma(t)}{t^{n} - 1}\right) \left(2 - t^{n-1}\right) \leq \frac{1}{n}
\end{equation}
and this follows from Lemma~\ref{lem_gamma}.
To prove the second point, we observe that $t \mapsto t^{-1} \left(\frac{t^n - 1}{t^{n-1} - 1}\right)$ is decreasing and that its limit when $t \to 1$ is $\frac{n}{n-1}$.

We finally prove that for all $\abs{x} \geq 1$ and for all $r,s \in [0,1]$, we have $\abs{\int_r^s \! \phi^x(x,t) \dm t} \leq \beta (r^2 + s^2)$.
Since $\phi^x = - \abs{\phi^x} e_r$, it amounts to show that for all $\abs{x} \geq 1$ and for all $s \in [0,1]$, we have
\begin{equation}\label{eq_int_phi_x}
    \int_0^s \! \abs{\phi^x(x,t)} \dm t \leq \beta s^2.
\end{equation}
If $\abs{x} \geq R$, inequality (\ref{eq_int_phi_x}) holds true for all $s \in [0,1]$ because we have $\abs{\phi^x(x,t)} \leq 2 \beta t$ for all $t \in [0,1]$.
Now, we fix $1 \leq \abs{x} < R$ and as usual, we write $r$ for $\abs{x}$.
Inequality (\ref{eq_int_phi_x}) holds true for all $s \in [0,\rho(r)]$ because we have $\abs{\phi^x(x,t)} \leq 2 \beta t$ for all $t \in [0,\rho(r)]$.
Next, we estimate for $t \in [\rho(r),1]$,
\begin{equation}
    \abs{\phi^x(x,t)} \leq 2 \abs{\nabla u(x)} = 2 \beta \delta(R) \left(\frac{R}{r}\right)^{n-1}
\end{equation}
whence for $s \in [\rho(r),1]$,
\begin{multline}
    \beta s^2 - \int_0^s \! \abs{\phi^x(x,t)} \dm t \geq \beta s^2 - \beta \delta(R)^2 \\ - 2 \beta \delta(R) (\rho(r) - \delta(R)) - 2 \beta \delta(R) \left(\frac{R}{r}\right)^{n-1} (s - \rho(r)).
\end{multline}
The right hand side function attains its minimum over $s \in \R$ at $s = \delta(R) \left(\frac{R}{r}\right)^{n-1}$ and its corresponding value is
\begin{equation}
    2 \beta \delta(R) \rho \left[\left(\frac{R}{r}\right)^{n-1} - 1\right] - \beta \delta(R)^2 \left[\left(\frac{R}{r}\right)^{2n - 2} - 1\right].
\end{equation}
It is non-negative if and only if
\begin{equation}
    \rho(r) \geq \frac{\delta(R)}{2} \left[\left(\frac{R}{r}\right)^{n-1} + 1\right].
\end{equation}
and given the formula (\ref{eq_rho}) of $\rho$, this means
\begin{multline}
    \beta r \left(\frac{\left(\tfrac{R}{r}\right)^{2n - 2} \Gamma\left(\tfrac{R}{r}\right)}{\left(\tfrac{R}{r}\right)^{n-1} - 1}\right) \\ - \frac{r}{n}\left(\beta - \frac{n-1}{R}\right) \left(\frac{\left(\tfrac{R}{r}\right)^n - 1}{\left(\tfrac{R}{r}\right)^{n-1} - 1}\right) \geq \left(\frac{R}{r}\right)^{n-1}.
\end{multline}
We rewrite this,
\begin{equation}\label{eq_gamma_estimate0}
    \beta r \left(\frac{\left(\tfrac{R}{r}\right)^{2n - 2} \Gamma\left(\tfrac{R}{r}\right) }{\left(\tfrac{R}{r}\right)^{n} - 1} - \frac{1}{n}\right) \geq \left(\frac{R}{r}\right)^{n-1} \left(\frac{\left(\tfrac{R}{r}\right)^{n-1} - 1}{\left(\tfrac{R}{r}\right)^{n} - 1}\right) - \frac{n-1}{n} \frac{r}{R}.
\end{equation}
We see first that the left-hand side is non-negative.
Indeed, for all $t \geq 1$,
\begin{equation}\label{eq_gamma_estimate00}
    \Gamma(t) \geq \frac{1}{n-1} \frac{t^{n-1} - 1}{t^{n-1}} \geq \frac{1}{n} \frac{t^n - 1}{t^n},
\end{equation}
where the first inequality comes from Lemma~\ref{lem_gamma} and the second one comes from the fact that whenever $\alpha > \beta > 0$ and $t \geq 1$,
\begin{equation}
    \frac{1}{\beta} t^{\alpha - \beta} (t^\beta - 1) \geq \frac{1}{\alpha} (t^\alpha - 1).
\end{equation}
Then, (\ref{eq_gamma_estimate0}) is equivalent to the fact that for all $t > 1$,
\begin{equation}\label{eq_gamma_estimate01}
    \beta \left(\frac{t^{2n - 2} \Gamma(t)}{t^n - 1} - \frac{1}{n}\right) \geq t^{n-1} \left(\frac{t^{n-1} - 1}{t^{n} - 1}\right) - \frac{n-1}{n} t^{-1}
\end{equation}
The inequality holds true if $\beta \geq n - \tfrac{1}{2}$ by the last point of Lemma~\ref{lem_gamma}.
In fact, it is necessary for $\beta$ to be greater than or equal to $n - \tfrac{1}{2}$ because dividing the left-hand side by $t - 1$ and taking the limit $t \to 1^+$ yields the value $\frac{n-1}{n} \beta$ whereas the same operation at the right-hand side yields $\frac{n-1}{n} \left(n - \tfrac{1}{2}\right)$.
We leave the details to the interested reader.

\begin{appendices}

    \section{The function \texorpdfstring{$\Gamma$}{}}\label{appendix_gamma}

    We recall and prove Lemma~\ref{lem_potential}.

    \begin{lem}
        Let $\Omega$ be a non-empty $C^2$ star-shaped bounded open set of $\R^n$.
        There exists a continuous function $\Gamma \colon \R^n \setminus \Omega \to \R$ such that
        \begin{enumerate}[label = (\roman*)]
            \item $\Gamma$ is harmonic positive on $\R^n \setminus \overline{\Omega}$;
            \item $\Gamma = 0$ on $\partial \Omega$,
            \item For all $x \in \R^n \setminus \overline{\Omega}$, $\nabla \Gamma(x) \ne 0$
            \item $\Gamma$ has a $C^1$ extension up to the boundary and for all $x \in \partial \Omega$, there exists $t > 0$ such that $\nabla \Gamma(x) = t \nu(x)$, where $\nu(x)$ is the outward unit normal vector of $\Omega$ at $x$.
        \end{enumerate}
    \end{lem}

    \begin{proof}
        Without loss of generality, we assume that $0 \in \Omega$ and that $\Omega$ is star-shaped with respect to $0$.
        We detail the case $n \geq 3$.
        According to~\cite[Section 3A, Theorem 3.40]{Fo}, there exists a unique function $p \in C(\R^n \setminus \Omega)$ such that
        \begin{enumerate}[label = (\roman*)]
            \item $p$ is harmonic on $\R^n \setminus \overline{\Omega}$,
            \item $p$ is harmonic at infinity,
            \item $p = 1$ on $\partial \Omega$.
        \end{enumerate}
        We refer to~\cite[Proposition 2.74]{Fo} for the characterisations of functions which are harmonic at infinity.
        Now we define $\Gamma(x) = 1 - p(x)$ and we review the properties of the lemma.
        It is clear that $\Gamma$ is harmonic on $\R^n \setminus \overline{\Omega}$ and $\Gamma = 0$ on $\partial \Omega$.
        As $p$ is harmonic at infinity, we have $\lim_{x \to \infty} p(x) = 0$ and thus $\lim_{x \to +\infty} \Gamma(x) = 1$.
        We can then apply the maximum principle to see that $\Gamma > 0$ on $\R^n \setminus \overline{\Omega}$.
        The function $\Gamma$ has a $C^1$ extension up to the boundary thanks to the usual regularity results for Dirichlet problems.
        As $\Gamma$ is constant on $\partial \Omega$, its tangential derivative is $0$ along $\partial \Omega$.
        And according to the Hopf Lemma, the normal derivative (with respect to the outward normal vector) is $> 0$ along $\partial \Omega$.
        This proves that for all $x \in \partial \Omega$, there exists $t > 0$ such that $\nabla \Gamma(x) = t \nu(x)$.
        The fact that $\nabla \Gamma$ never vanishes comes from the fact that $\Omega$ is star-shaped.
        Indeed, the function $w \colon x \mapsto x \cdot \nabla \Gamma(x)$ is harmonic on $\R^n \setminus \overline{\Omega}$ and $\geq 0$ on $\partial \Omega$.
        In addition, we see that $\lim_{x \to +\infty} w(x) = 0$ by applying~\cite[Proposition 2.75]{Fo} to the function $p$.
        We can use the maximum principle to conclude that $w > 0$ on $\R^n \setminus \Omega$

        In the case $n = 2$, we define $p$ as the unique function $p \in C(\R^n \setminus \Omega)$ such that
        \begin{enumerate}[label = (\roman*)]
            \item $p$ is harmonic on $\R^n \setminus \overline{\Omega}$,
            \item $p$ is harmonic at infinity,
            \item $p(x) = \ln(\abs{x})$ on $\partial \Omega$.
        \end{enumerate}Then we define $\Gamma = \ln(\abs{x}) - p$.
        As $p$ is harmonic at infinity, it is bounded at infinity and thus $\lim_{x \to +\infty} \Gamma(x) = +\infty$.
        The rest of the proof is the same except that $\lim_{x \to +\infty} x \cdot \nabla \Gamma(x) = 1$.
        The case $n = 1$ is trivial.
    \end{proof}

    In the case of the unit ball $\Omega = B(0,1)$, the function $\Gamma$ is given by
    \begin{equation}\label{eq_gamma3}
        \Gamma(r) =
        \begin{cases}
            r - 1                      & \ \text{if $n = 1$} \\
            \ln(r)                     & \ \text{if $n = 2$} \\
            \frac{1}{n-2}(1 - r^{2-n}) & \ \text{if $n \geq 3$}.
        \end{cases}
    \end{equation}

    We isolate a few useful estimates about this function in the following lemma.
    The reader is free to skip the proof.

    \begin{lem}\label{lem_gamma}
        Let $n$ be an integer $\geq 2$ and let $\Gamma$ be defined as in (\ref{eq_gamma3}).
        \begin{enumerate}[label = (\roman*)]
            \item For all $s, t \geq 1$, $\Gamma(t) - \Gamma(s) = s^{2 - n} \Gamma\left(\frac{t}{s}\right)$.
            \item The function
                \begin{equation}
                    t \mapsto \frac{t^{n-1} \Gamma(t)}{t^{n-1} - 1}
                \end{equation}
                is increasing on $]1,+\infty[$.
            \item For all $t \geq 1$,
                \begin{equation}\label{eq_gamma_estimate}
                    \frac{1}{n-1} \frac{t^{n-1} - 1}{t^{n-1}} \leq \Gamma(t) \leq \frac{1}{n} \frac{t^n - 1}{t^{n-1}}.
                \end{equation}
            \item For all $t > 1$,
                \begin{equation}\label{eq_gamma_estimate2}
                    \left(n - \tfrac{1}{2}\right) \left(\frac{t^{2n - 2} \Gamma(t)}{t^n - 1} - \frac{1}{n}\right) \geq t^{n-1} \left(\frac{t^{n-1} - 1}{t^{n} - 1}\right) - \frac{n-1}{n} t^{-1}
                \end{equation}
        \end{enumerate}
    \end{lem}

    \begin{proof}
        The first point is easy and we pass directly to the second one.
        If $n = 2$, the function
        \begin{equation}
            t \mapsto \frac{t^{n-1} \Gamma(t)}{t^{n-1} - 1} = \frac{t \ln(t)}{t - 1}
        \end{equation}
        is increasing because $t \mapsto t \ln(t)$ is convex.
        If $n \geq 3$, the function
        \begin{equation}
            t \mapsto \frac{t^{n-1} \Gamma(t)}{t^{n-1} - 1} = \frac{1}{n - 2} \frac{t^{n-1} - t}{t^{n-1} - 1}
        \end{equation}
        is increasing because
        \begin{equation}
            \frac{t^{n-1} - t}{t^{n-1} - 1} = 1 - \frac{t - 1}{t^{n-1} - 1}
        \end{equation}
        and $t \mapsto t^{n-1}$ is convex.
        Since $\lim_{t \to 1} \frac{t^{n-1} \Gamma(t)}{t^{n-1} - 1} = \frac{1}{n-1}$,
        we also deduce that for all $t \geq 1$,
        \begin{equation}
            \Gamma(t) \geq \frac{1}{n-1} \frac{t^{n-1} - 1}{t^{n-1}}.
        \end{equation}
        Next, we prove that for all $t \geq 1$,
        \begin{equation}
            \Gamma(t) \leq \frac{1}{n} \frac{t^n - 1}{t^{n-1}}.
        \end{equation}
        We rewrite this condition
        \begin{equation}
            n t^{n-1} \Gamma(t) \leq t^n - 1
        \end{equation}
        and since both sides equals $0$ at $t = 1$, it suffices to show that the derivative of the left-hand side is greater than or equal to the derivative of the left-hand side on $[1,+\infty[$.
        We are thus led to show that for all $t \geq 1$,
        \begin{equation}
            n (n - 1) t^{n - 2} \Gamma(t) + n \leq n t^{n-1},
        \end{equation}
        that is
        \begin{equation}
            \Gamma(t) \leq \frac{1}{n-1} \frac{t^{n-1} - 1}{t^{n - 2}}.
        \end{equation}
        When $n = 2$, this comes from the concavity of $t \mapsto \ln(t)$.
        When $n \geq 3$, this comes from the fact that whenever $\alpha > \beta > 0$ and $t \geq 1$,
        \begin{equation}
            \beta^{-1} (t^\beta - 1) \leq \alpha^{-1} (t^\alpha - 1).
        \end{equation}
        We finally show that for all $t > 1$,
        \begin{equation}\label{eq_gamma_estimate3}
            \left(n - \tfrac{1}{2}\right) \left(\frac{t^{2n - 2} \Gamma(t)}{t^n - 1} - \frac{1}{n}\right) \geq t^{n-1} \left(\frac{t^{n-1} - 1}{t^{n} - 1}\right) - \frac{n-1}{n} t^{-1}.
        \end{equation}
        We isolate $\Gamma$ in (\ref{eq_gamma_estimate3}) and we obtain the equivalent condition
        \begin{equation}
            n \left(n - \tfrac{1}{2}\right) t^{2n - 1} \Gamma(t) \geq n t^{2n - 1} + \left(n - \tfrac{1}{2}\right) \left(t^{n+1} - 2t^n - t\right) + \tfrac{1}{2} (n - 1).
        \end{equation}
        Since both sides equals $0$ at $t = 1$, it suffices to show that the derivative of the left-hand side is greater than or equal to the derivative of the left-hand side on $[1,+\infty[$.
        This condition simplifies to the fact that for all $t \geq 1$,
        \begin{equation}\label{eq_gamma_computations}
            n \left(n - \tfrac{1}{2}\right) t^{2n - 2} \Gamma(t) \geq n t^{n-1} (t^{n-1} - 1) + \tfrac{1}{2} (t^n - 1).
        \end{equation}
        The condition (\ref{eq_gamma_computations}) holds true when we replace $\Gamma(t)$ by the lower bound
        \begin{equation}
            \Gamma(t) \geq \frac{1}{n-1} \frac{t^{n-1} - 1}{t^{n-1}}.
        \end{equation}
        Indeed, the new inequality simplifies to the fact that for all $t \geq 1$,
        \begin{equation}
            \frac{1}{n-1} t^{n-1} (t^{n-1} - 1) \geq \frac{1}{n} (t^n - 1)
        \end{equation}
        and this comes from the fact that whenever $\alpha > \beta > 0$ and $t \geq 1$,
        \begin{equation}
            \frac{1}{\beta} t^{\alpha - \beta} (t^\beta - 1) \geq \frac{1}{\alpha} (t^\alpha - 1).
        \end{equation}
    \end{proof}

    %
    %

    \section{Calibrations}\label{appendix_calibration}

    The goal of this section is to recall the main results and definitions of~\cite{Alberti} which explain why the existence of a vectorfield as in Definitions~\ref{defi_harmonic_calibration} and~\ref{defi_isolation_calibration} implies the minimality of $u$.

    Let $X$ be an open set of $\R^n$ and let $u \in SBV(X)$.
    We let $E_u$ denote the \emph{subgraph} of $u$
    \begin{equation}
        E_u = \set{(x,t) \in X \times \R | t \leq u(x)}.
    \end{equation}
    It has a locally finite perimeter in $X \times \R$.
    We define the \emph{complete graph} of $u$, written $\Gamma_u$, as the measure-theoretic boundary of $E_u$.
    According to the usual structure theorem, $D \mathbf{1}_{E_u} = \nu_{\Gamma_u} \HH^{n-1} \mres \Gamma_u$ where $\nu_{\Gamma_u} \in \R^n \times \R$ is the measure-theoretic inward normal to $E_u$. This notation should not be confused with $\nu_u \in \R^n$ which is defined on $J_u$.

    We will rely on~\cite[Lemma 2.10, Lemma 2.12]{Alberti}, summarized in Lemma~\ref{lem_calibration_tools} below.
    We refer to~\cite[Definition 2.1]{Alberti} for the definition of an approximately regular vector field.
    The reader can think of it as a piecewise continuous vector field $\phi$ whose normal component does not jump through the discontinuity, except on a $\HH^{n-1}$-negligible set.
    The point of this property is to give a pointwise meaning (except on a $\HH^{n-1}$-negligible set) to the "normal component of $\phi$" in the divergence theorem.
    We are more specifically interested in bounded approximately regular vector fields which are divergence-free in the sense of distribution.
    The main example is a piecewise smooth vector field which is divergence-free on each separate piece and whose normal component does not jump through the discontinuity.
    See~\cite[Lemma 2.6 and Remark 2.8]{Alberti} for a more formal statement.

    \begin{lem}\label{lem_calibration_tools}
        Let $X$ be an open set of $\R^n$, let $u \in SBV(X)$.
        \begin{enumerate}
            \item For any Borel map $\phi \colon X \times \R \to \R^n \times \R$, we have
                \begin{multline}\label{eq_lem_calibration_tool1}
                    \int_{X \times \R} \! \phi \cdot Du = \int_X \! \phi^x(x,u) \cdot \nabla u(x) - \phi^t(x,u) \dm x \\ + \int_{J_u} \int_{u^-}^{u^+} \! \phi^x(x,t) \cdot \nu_u(x) \dm t \dm \HH^{n-1}(x),
                \end{multline}
                provided that the right-hand side is well-defined (it suffices that $\phi$ is bounded and $x \mapsto \phi^t(x,u(x))$ has a bounded support).
            \item We assume that $X$ is a bounded open set of Lipschitz boundary and we let $\nu_X$ denote its inner normal vector field (defined $\HH^{n-1}$-a.e.\ on $\partial X$).
                Let $m \leq M$ be two real numbers, let $\phi \colon X \times [m,M] \to \R^n \times \R$ be a Borel map which is bounded, approximately regular in $X \times [m,M]$ and divergence-free on $X \times ]m,M[$ in the sense of distribution.
                Then for all $u, v \in SBV(X)$ such that $m \leq u,v \leq M$,
                \begin{multline}
                    \int_{\Gamma_u} \! \phi \cdot \nu_{\Gamma_u} \dm \HH^{n-1} - \int_{\Gamma_v} \! \phi \cdot \nu_{E_v} \dm \HH^{n-1} \\ = \int_{\partial X} \int_{u^*}^{v^*} \! \phi^x(x,t) \cdot \nu_X(x) \dm t \dm \HH^{n-1}(x),
                \end{multline}
                where $u^*$, $v^*$ are the traces of $u$ and $v$ on $\partial X$.
        \end{enumerate}
    \end{lem}

    The general principle of calibrations for free-discontinuity problems is as follow.
    Let us say that we minimize an energy $E(u)$ over functions $u \in SBV(X)$.
    Given a candidate minimizer $u \in SBV(X)$, we expect a calibration for $u$ to satisfy the following properties: for all competitor $v$ of $u$, we have
    \begin{subequations}\label{eq_calibration_principles}
        \begin{equation}
            \int_{\Gamma_v} \! \phi \cdot \nu_{\Gamma_v} \dm \HH^{n-1} = \int_{\Gamma_u} \! \phi \cdot \nu_{\Gamma_u} \dm \HH^{n-1}
        \end{equation}
        and
        \begin{equation}
            \int_{\Gamma_v} \! \phi \cdot \nu_{\Gamma_v} \dm \HH^{n-1} \leq E(v)
        \end{equation}
        with equality when $v = u$.
    \end{subequations}
    This clearly implies the minimality of $u$.

    We consider $A$, $V$, $u$, $M$, $\phi$ as in Definition~\ref{defi_harmonic_calibration}.
    We extend $\phi$ by $\phi(x,t) = 0$ outside $\overline{A} \times [m,M]$.
    The extension may not be approximately regular and divergence-free outside $\overline{A} \times [m,M]$ but it does not matter.
    We check the two conditions (\ref{eq_calibration_principles}) which in our case means $X = V$ and for all competitor $v \in SBV(V)$ such that $0 \leq v \leq M$ and $v = u$ in $V \setminus \overline{A}$, we have
    \begin{subequations}
        \begin{equation}
            \int_{\Gamma_v \cap \overline{A}} \! \phi \cdot \nu_{\Gamma_v} \dm \HH^{n-1} = \int_{\Gamma_u \cap \overline{A}} \! \phi \cdot \nu_{\Gamma_u} \dm \HH^{n-1}\label{eq_calibration_goal1}
        \end{equation}
        and
        \begin{equation}
            \int_{\Gamma_v \cap \overline{A}} \! \phi \cdot \nu_{\Gamma_v} \dm \HH^{n-1} \leq E_0(v,\overline{A}),\label{eq_calibration_goal2}
        \end{equation}
        with equality when $v = u$.
    \end{subequations}
    Let us start with (\ref{eq_calibration_goal2}).
    We apply the first item of Lemma~\ref{lem_calibration_tools} with $X = V$.
    We have thus for all $v \in SBV(V)$ such that $0 \leq v \leq M$,
    \begin{multline}\label{eq_calibration_step1}
        \int_{\Gamma_v \cap \overline{A}} \! \phi \cdot \nu_{\Gamma_v} \dm \HH^{n-1} = \int_A \! \phi^x(x,v) \cdot \nabla v(x) - \phi^t(x,v) \dm x \\ + \int_{J_v \cap \overline{A}} \int_{v^-}^{v^+} \! \phi^x(x,t) \cdot \nu_v(x) \dm t \dm \HH^{n-1}(x).
    \end{multline}
    The requirements (a) and (b) allow to bound the right-hand side of (\ref{eq_calibration_step1}) and to obtain
    \begin{equation}\label{eq_calibration_step2}
        \int_{\Gamma_v \cap \overline{A}} \! \phi \cdot \nu_{\Gamma_v} \dm \HH^{n-1} \leq \int_A \! \abs{\nabla v}^2 \dm x + \beta \int_{J_v \cap \overline{A}} \! (v^+)^2 + (v^-)^2 \dm \HH^{n-1},
    \end{equation}
    with equality for $v = u$.
    We have proved (\ref{eq_calibration_goal2}).

    \begin{rmk}\label{rmk_calibration_equality}
        One can also see, using the requirements (a) and (b), that there is equality in (\ref{eq_calibration_step2}) if and only if
        \begin{enumerate}
            \item[(a')] $\phi^x(x,v) = 2 \nabla v$ and $\phi^t(x,v) = \abs{\nabla v}^2$ for $\LL^n$-a.e.\ $x \in A$,
            \item[(b')] $\int_{v^-}^{v^+} \! \phi^x(x,t) \dm t = \beta \left[(v^-)^2 + (v^+)^2\right]\nu_v$ for $\HH^{n-1}$-a.e.\ $x \in J_v \cap \overline{A}$.
        \end{enumerate}
        We deduce that (a'), (b') holds true not only for $v = u$ but for all minimizers, that is for all competitors $v$ such that $E_0(v,\overline{A}) = E_0(u,\overline{A})$.
    \end{rmk}

    We pass to (\ref{eq_calibration_goal1}).
    We apply the second item of Lemma~\ref{lem_calibration_tools} with $X = A$.
    For all $v \in SBV(V)$ such that $0 \leq v \leq M$, we have
    \begin{multline}\label{eq_calibration_step3}
        \int_{\Gamma_u \cap A} \! \phi \cdot \nu_{\Gamma_u} \dm \HH^{n-1} = \int_{\Gamma_v \cap A} \! \phi \cdot \nu_{\Gamma_v} \dm \HH^{n-1} \\ + \int_{\partial A} \int_{u^*}^{v^*} \! \phi^x(x,t) \cdot \nu_A(x) \dm t \dm \HH^{n-1}(x),
    \end{multline}
    where $u^*$ and $v^*$ are the traces of $u\vert_A$ and $v\vert_A$ on $\partial A$.
    Using the fact that $A$ has a Lipschitz boundary and assuming $v = u$ in $V \setminus \overline{A}$, we are going to show that
    \begin{equation}\label{eq_uv_star_pm}
        \int_{u^*}^{v^*} \! \phi^x(x,t) \cdot \nu_A(x) \dm t = \int_{v^-}^{v^+} \! \phi^x(x,t) \cdot \nu_v(x) \dm t - \int_{u^+}^{u^-} \! \phi^x(x,t) \cdot \nu_u \dm t.
    \end{equation}
    For $\HH^{n-1}$-a.e.\ $x \in A$, the function $v$ has a trace at each side of $\partial A$; the inner trace $v^*$ that we have already defined and an outer trace, that we shall write $v_*$.
    Moreover, for $\HH^{n-1}$-a.e.\ $x \in \partial A$, we have the following alternative.
    Either $x$ is a Lebesgue point of $v$ and $v^+(x) = v^-(x) = v^*(x) = v_*(x)$, or $x$ is a jump point of $A$ with $\nu_v(x) = \pm \nu_A(x)$, $v^+(x) = \max(v^*(x),v_*(x))$ and $v^- = \min(v^*(x),v_*(x))$.
    The vector $\nu_v$ points toward the higher trace of $v$ and $\nu_A$ is an inner normal to $A$ so the case $\nu_v(x) = \nu_A(x)$ corresponds to the matching $v^+(x) = v^*(x)$, $v^-(x) = v_*(x)$ and the case $\nu_v(x) = - \nu_A(x)$ corresponds to the matching $v^+(x) = v_*(x)$, $v^-(x) = v^*(x)$.
    This proves that
    \begin{equation}\label{eq_v_star_pm}
        \int_{v_*}^{v^*} \! \phi^x(x,t) \cdot \nu_A(x) \dm t = \int_{v^-}^{v^+} \! \phi^x(x,t) \cdot \nu_v(x) \dm t.
    \end{equation}
    We have similarly,
    \begin{equation}\label{eq_u_star_pm}
        \int_{u_*}^{u^*} \! \phi^x(x,t) \cdot \nu_A(x) \dm t = \int_{u^-}^{u^+} \! \phi^x(x,t) \cdot \nu_u(x) \dm t.
    \end{equation}
    And since $u = v$ in $V \setminus A$, we have $u_*(x) = v_*(x)$ for $\HH^{n-1}$-a.e.\ $x \in \partial A$.
    Now, it suffices to substract (\ref{eq_u_star_pm}) from (\ref{eq_v_star_pm}) to obtain (\ref{eq_uv_star_pm}).
    The equality (\ref{eq_calibration_step3}) can be rewritten
    \begin{multline}\label{eq_calibration_step4}
        \int_{\Gamma_u \cap A} \! \phi \cdot \nu_{\Gamma_u} \dm \HH^{n-1} + \int_{J_u \cap \partial A} \int_{u^-}^{u^+} \! \phi^x(x,t) \cdot \nu_u(x) \dm t \dm \HH^{n-1}(x) \\ = \int_{\Gamma_v \cap A} \! \phi \cdot \nu_{\Gamma_v} \dm \HH^{n-1} + \int_{J_v \cap \partial A} \int_{v^-}^{v^+} \! \phi^x(x,t) \cdot \nu_v(x) \dm t \dm \HH^{n-1}(x)
    \end{multline}
    Apply (\ref{eq_lem_calibration_tool1}) with $X = A$ to develop $\int_{\Gamma_u \cap A} \! \phi \cdot \nu_{\Gamma_u} \dm \HH^{n-1}$. In view of (\ref{eq_calibration_step1}), we recognize that the left-hand side of (\ref{eq_calibration_step4}) is $\int_{\Gamma_u \cap \overline{A}} \! \phi \cdot \nu_{\Gamma_u} \dm \HH^{n-1}$. We reason similarly with the right-hand side and  we have proved (\ref{eq_calibration_goal1}).

    We would proceed in the same way with Definition~\ref{defi_isolation_calibration}.
    We recall that there exists $\delta > 0$ (depending only on $\Omega$, $\beta$, $\gamma$) such that $\Omega \subset \subset B(0,\delta^{-1})$ and all relevant competitors of (\ref{eq_E}) have a compact support included in $B(0,\delta^{-1})$.
    We would apply the items of Lemma~\ref{lem_calibration_tools} with $X = B(0,\delta^{-1})$ and $X = B(0,\delta^{-1}) \setminus \overline{\Omega}$.

    \section{Extension of a unit normal vector field}\label{appendix_extension}

    Let $X$ be an open set of $\R^n$, let $\Omega \subset \subset X$ be an open subset of $X$ with Lipschitz boundary.
    We let $\nu$ denote the outward unit normal vector field of $\Omega$ (defined $\HH^{n-1}$-a.e.\ on $\partial \Omega$), and we define $A = X \setminus \overline{\Omega}$.

    We recall the functional definition of the "normal component on $\partial \Omega$" introduced by Anzellotti (\cite{Anzellotti}).
    Let $z \in L^\infty(A;\R^n)$ be a vector field such that $\mathrm{div}(z)$ is a finite Radon measure in $A$.
    According to~\cite[Theorem 1.2]{Anzellotti}, there exists a unique function $[z \cdot \nu] \in L^\infty(\partial \Omega)$ such that for all $\varphi \in C^1_c(X)$,
    \begin{equation}\label{eq_normal_component}
        \int_A \! \varphi \dm [\mathrm{div}(z)] - \int_A \! z \cdot \nabla \varphi \dm x = -\int_{\partial \Omega} \! [z \cdot \nu] \varphi \dm \HH^{n-1}.
    \end{equation}
    Moreover, we have $\norm{[z,\cdot \nu]}_{L^\infty(\partial \Omega)} \leq \norm{z}_{L^\infty}$.
    If $z$ was a smooth vectorfield, the function $[z \cdot \nu]$ would coincide with the scalar product $z \cdot \nu$ on $\partial \Omega$.

    In the case $\mathrm{div}(z) = 0$ in $A$, we can give another interpretation to (\ref{eq_normal_component}).
    We extend $z$ by $0$ in $\Omega$ so (\ref{eq_normal_component}) says that for all $\varphi \in C^1_c(X)$,
    \begin{equation}
        \int_X \! z \cdot \nabla \varphi \dm x = -\int_{\partial \Omega} \! [z \cdot \nu] \varphi \dm \HH^{n-1}.
    \end{equation}
    Thus, $\mathrm{div}(z)$ is a finite Radon measure in $X$ and is equal to $[z \cdot \nu] \HH^{n-1} \mres \partial \Omega$.
    As $\Omega$ has a Lipschitz boundary, $z = 0$ in $\Omega$ and $\mathrm{div}(z)$ is supported on $\partial \Omega$, one can see the pairing $(z, D \mathbf{1}_\Omega^+)$ (see \cite[Definition 3.1]{bv_supersolutions}) coincides with $-\mathrm{div}(z)$, i.e.
    \begin{equation}
        (z, D \mathbf{1}_\Omega^+) = - [z \cdot \nu] \norm{D \mathbf{1}_\Omega}.
    \end{equation}
    This means that
    \begin{equation}
        z = - [z \cdot \nu] \frac{D \mathbf{1}_\Omega}{\norm{D \mathbf{1}_\Omega}}  
    \end{equation}
    in a weak sense.

    Now, we assume that $\Omega$ is outward minimizing in $X$; for all set of finite perimeter $E \subset \subset X$ containing $\Omega$, we have $P(\Omega) \leq P(E)$.
    This property holds true for convex sets but not only.
    If the curvature of $\partial \Omega$ is positive, there exists an open neighborhood $X$ of $\Omega$ in which $\Omega$ is outward minimizing.
    According to the co-area formula, we deduce that that for all $\varphi \in C^1_c(X)$ such that $\varphi \geq 1$ on $\Omega$,
    \begin{equation}\label{eq_outward_minimizing}
        P(\Omega) \leq \int_X \! \abs{\nabla \varphi} \dm x.
    \end{equation}
    Inequality (\ref{eq_outward_minimizing}) also holds true all $u \in W^{1,1}_0(X)$ such that $u \geq 1$ on $\Omega$.
    Indeed, one can first replace $u$ by its post-composition with the orthogonal projection onto $[0,1]$.
    This makes sure that $u = 1$ and $\Omega$ and do not increase the total variation.
    Since $\Omega$ has a Lipschitz boundary, the function $u$ has trace $1$ on $\partial \Omega$ and can be approximated in $W^{1,1}(X)$ convergence by functions $\varphi \in C^1_c(X)$ such that such that $\Omega \subset \subset \set{\varphi = 1}$.

    The following lemma and its elementary proof were communicated by A. Chambolle.
    There is an alternative proof, in a more general setting, in~\cite[Appendix A]{C5}.

    \begin{prop}
        Let $X$ be a bounded open set of $\R^n$, let $\Omega \subset \subset X$ be an open subset of $X$ with Lipschitz boundary.
        We let $\nu$ denote the outward unit normal vector field of $\Omega$ (defined $\HH^{n-1}$-a.e.\ on $\partial \Omega$), and $A = X \setminus \overline{\Omega}$.
        Then there exists a vectorfield $z \in L^\infty(A;\R^n)$ such that $\norm{z}_{L^\infty} \leq 1$, $\mathrm{div}(z) = 0$ in $A$ and $[z \cdot \nu] = 1$ $\HH^{n-1}$-a.e.\ on $\Omega$, that is for all $\varphi \in C^1_c(X)$,
        \begin{equation}
            \int_A \! z \cdot \nabla \varphi \dm x = - \int_{\partial \Omega} \! \varphi \dm \HH^{n-1}.
        \end{equation}
    \end{prop}

    \begin{proof}
        Without loss of generality, we assume $\Omega \ne \emptyset$.
        Let $p$ be a real number such that $1 < p < 2$.
        We write $p'$ the positive number such that $\frac{1}{p} + \frac{1}{p'} = 1$, i.e.\ $p' = p / (p - 1)$.
        We let $\mathcal{S}_p$ denote the affine subsace
        \begin{equation}
            \set{u \in W^{1,p}_0(X) | u = 1 \ \text{on $\Omega$}}.
        \end{equation}
        As we have seen before, the elements of $\mathcal{S}_p$ can be approximated in $W^{1,p}$ convergence by functions $\varphi \in C^1_c(X)$ such that $\Omega \subset \subset \set{\varphi = 1}$.
        We can also approximate $\mathbf{1}_\Omega$ in strict $BV$ convergence (\cite[Definition 3.14]{Ambrosio}) by such functions.
        Indeed, $\Omega$ has a Lipschitz boundary so for all $\varepsilon$, there exists $\delta > 0$ such that $\Omega_\delta \subset \subset X$ and $P(\Omega_\delta) \leq P(\Omega) + \varepsilon$, where $\Omega_\delta = \set{x \in \R^n | \mathrm{dist}(x,\Omega) < \delta}$.
        We can also assume $\delta$ small enough so that $\abs{\Omega_\delta \setminus \Omega} \leq \varepsilon$.
        Then, it suffices to mollify $\mathbf{1}_{\Omega_\delta}$.

        We consider the energy
        \begin{equation}
            E_p(u) := \int_{\R^n} \! \abs{\nabla u}^p \dm x
        \end{equation}
        defined for $u \in \mathcal{S}_p$ and we let $m_p$ denote its infimum.

        We can bound $m_p$ from above independently from $p$, by considering any test function $\varphi \in C^1_c(X)$ such that $\varphi = 1$ on $\Omega$ and $\abs{\nabla \varphi} \leq C$ on $\R^n$, where $C$ depends on $\Omega$ and $X$.

        Next, we show that $\lim_{p \to 1} m_p = P(\Omega)$.
        We start with $\limsup_{p \to 1} m_p \leq P(\Omega)$.
        For all $\varphi \in C^1_c(X)$ such that $\varphi = 1$ on $\Omega$, we have
        \begin{equation}
            m_p \leq \int_{X} \! \abs{\nabla \varphi}^p \dm x
        \end{equation}
        whence
        \begin{equation}
            \limsup_{p \to 1} m_p \leq \int_{X} \! \abs{\nabla \varphi} \dm x.
        \end{equation}
        We deduce that $\limsup_{p \to 1} m_p \leq P(\Omega)$ by approximating $\mathbf{1}_\Omega$ in strict $BV$ convergence with functions $\varphi \in C^1_c(X)$ such that $\varphi = 1$ on $\Omega$.
        Reciprocally, for all $u \in \mathcal{S}_p$, the outward minimality of $\Omega$ yields
        \begin{align}
            P(\Omega) & \leq \int_{X} \! \abs{\nabla u} \dm x \\
                      & \leq \abs{X}^{1 - \frac{1}{p}} \int_{X} \! \abs{\nabla u}^p \dm x
        \end{align}
        whence $P(\Omega) \leq \abs{X}^{1 - \frac{1}{p}} m_p$ and then $P(\Omega) \leq \liminf_{p \to 1} m_p$.

        The energy $E_p$ has a unique minimizer $u_p \in \mathcal{S}_p$.
        The short and usual proof starts with the fact that for any miniming sequence $(v_k)$, the gradient sequence $(\nabla v_k)$ is a Cauchy sequence in $L^p(X)$ thanks to the uniform convexity of the $L^p$ norm.
        One can use Poincar\'e inequality to deduce that for all open ball $B \subset X$, the sequence $(v_k - m_B v_k)$ is also a Cauchy sequence in $L^p(B)$, where $m_B v_k$ is the average value of $v_k$ in $B$.
        Similarly, for all open balls $B_1, B_2 \subset X$, the sequence of real numbers $(m_{B_1} v_k - m_{B_2})$ is also a Cauchy sequence.
        So we can fix a ball $B_0 \subset \Omega$ and see that $(v_k - m_{B_0} v_k)$ converge in $L^1_{\mathrm{loc}}(X)$, but since $v_k = 1$ on $\Omega$, this just means that $v_k$ converge in $L^1_{\mathrm{loc}}$.
        The limit is a minimizer and is even unique, again by uniform convexity of the $L^p$ norm.

        Now, we state the Euler-Lagrange associated to this problem.
        For all function $v \in \mathcal{S}_p$, we can test $v_\varepsilon = u_p \pm \varepsilon (v - u)$ for all $\varepsilon > 0$ to deduce
        \begin{equation}
            \int_X \! \abs{\nabla u_p}^{p - 2} \nabla u_p \cdot \nabla (v - u) \dm x = 0,
        \end{equation}
        that is
        \begin{equation}\label{eq_euler_zp1}
            \int_X \! \abs{\nabla u_p}^{p - 2} \nabla u_p \cdot \nabla v = \int_X \! \abs{\nabla u_p}^p \dm x.
        \end{equation}
        An equivalent way to state (\ref{eq_euler_zp1}) is that for all $v \in W^{1,p}_0(X)$ such that $v = 0$ on $\Omega$, we have
        \begin{equation}\label{eq_euler_zp2}
            \int_X \! \abs{\nabla u_p}^{p - 2} \nabla u_p \cdot \nabla v \dm x = 0.
        \end{equation}

        We set $z_p = -\abs{\nabla u}^{p - 2} \nabla u \in L^{p'}(\R^n)$ and we are going to extract a subsequence of $(z_p)$ that converges to a vector field $z$ when $p \to 1$; it will be the solution to our lemma.
        We recall that $p' = p / (p - 1)$ so $p' \to +\infty$ as $p \to 1$.
        Fix any $1 < q < \infty$ and consider $p$ close enough to $1$ so that $q \leq p'$.
        Then
        \begin{align}
            \left(\fint_X \abs{z_p}^q \dm x\right)^\frac{1}{q} & \leq \left(\fint_X \abs{z_p}^{p'} \dm x\right)^\frac{1}{p'} \\
                                                               & \leq \left(\fint_X \abs{\nabla u_p}^{p} \dm x\right)^\frac{1}{p'}
        \end{align}
        whence
        \begin{equation}
            \limsup_{p \to 1} \left(\fint_X \abs{z_p}^q \dm x\right)^\frac{1}{q} \leq 1.
        \end{equation}
        We use a diagonalisation argument to extract a subsequence of $(z_p)$ that converges weakly in $L^q$ to a vector field $z$, for all $1 < q < \infty$.
        The norm $\norm{\cdot}_{L^q}$ is lower-semicontinuous with respect to weak convergence in $L^q$ (since it can be computed by duality) so for all $1 < q < \infty$,
        \begin{equation}
            \left(\fint_X \abs{z}^q \dm x\right)^\frac{1}{q} \leq 1
        \end{equation}
        which means that $\abs{z} \leq 1$ on $X$.
        Passing to the limit in (\ref{eq_euler_zp1}) and (\ref{eq_euler_zp2}), we see that for all $\varphi \in C^1_c(X)$ with $\varphi = 1$ on $\Omega$, we have
        \begin{equation}\label{eq_euler_z1}
            \int_X \! z \cdot \nabla \varphi \dm x = -P(\Omega)
        \end{equation}
        and for all $\varphi \in C^1_c(X)$ with $\varphi = 0$ on $\Omega$, we have
        \begin{equation}\label{eq_euler_z2}
            \int_X \! z \cdot \nabla \varphi \dm x = 0.
        \end{equation}
        Observe that $\abs{\nabla u_p}^{p - 2} \nabla u_p = 0$ on $\Omega$ so $z = 0$ as well on $\Omega$.
        The condition (\ref{eq_euler_z2}) implies that $\mathrm{div}(z)$ is zero on $A$.
        According to the introduction of this section, there exists a function $[z \cdot \nu] \in L^\infty(\partial \Omega)$ such that $\norm{[z \cdot \nu]}_{L^\infty(\partial \Omega)} \leq 1$ and for all $\varphi \in C^1_c(X)$,
        \begin{equation}
            \int_X \! z \cdot \nabla \varphi \dm x = -\int_{\partial \Omega} \! [z \cdot \nu] \varphi \dm \HH^{n-1}.
        \end{equation}
        Taking any $\varphi \in C^1_c(X)$ such that $\varphi = 1$ on $\Omega$ yields
        \begin{equation}
            \int_{\partial \Omega} \! [z \cdot \nu] \dm \HH^{n-1} = P(\Omega)
        \end{equation}
        but since, $[z \cdot \nu] \leq 1$ $\HH^{n-1}$-a.e.\ on $\partial \Omega$, we actually have $[z \cdot \nu] = 1$ $\HH^{n-1}$-a.e\ on $\partial \Omega$.
    \end{proof}

\end{appendices}


\begin{thebibliography}{99}

    \bibitem{Alberti} G. Alberti, G. Bouchitt\'e and G. Dal Maso, The calibration
    method for the Mumford-Shah functional and free-discontinuity problems.
    Calc. Var. Partial Differential Equations 16 (2003), no. 3, 299-333.

    \bibitem{AC} H. W. Alt and L. A. Caffarelli, Existence and regularity for a
    minimum problem with free boundary. J. Reine Angew. Math., 325:105-144,
    1981.

    \bibitem{Ambrosio} L. Ambrosio, N. Fusco and D. Pallara, Functions of
    bounded variation and free discontinuity problems. Oxford Mathematical
    Monographs. The Clarendon Press, Oxford University Press, New York, 2000.
    xviii+434 pp. ISBN: 0-19-850245-1

    \bibitem{Anzellotti} G. Anzellotti, Pairings between measures and bounded
    functions and compensated compactness. Ann. Mat. Pura Appl. (4) 135 (1983),
    293-318 (1984).

    \bibitem{C2} H. Bischof, A. Chambolle, D. Cremers and T. Pock, An algorithm
    for minimizing the Mumford-Shah functional. 2009 IEEE 12th International
    Conference on Computer Vision.

    \bibitem{BF} B. Bogosel and M. Foare, Numerical implementation in 1D and 2D
    of a shape optimization problem with Robin boundary conditions. Preprint
    available at
    http://www.cmap.polytechnique.fr/~beniamin.bogosel/pdfs/Robin.pdf.

    \bibitem{Bo} G. Bouchitt\'e and I. Fragal\'a, A duality theory for non-convex
    problems in the calculus of variations. Arch. Ration. Mech. Anal. 229
    (2018), no. 1, 361-415.

    \bibitem{B2} D. Bucur. G. Buttazzo and C. Nitsch, Two optimization problems
    in thermal insulation. Notices Amer. Math. Soc. 64 (2017), no. 8, 830-835.

    \bibitem{B3} D. Bucur, G. Buttazzo and C. Nitsch, Symmetry breaking for a
    problem in optimal insulation. J. Math. Pures Appl. (9) 107 (2017), no. 4,
    451-463.

    \bibitem{BG} D. Bucur and A. Giacomini, Shape optimization problems with
    Robin conditions on the free boundary. Ann. Inst. H. Poincar\'e Anal. Non
    Lin\'eaire 33 (2016), no. 6, 1539-1568.
    
    \bibitem{BL} D. Bucur and S. Luckhaus, Monotonicity formula and regularity
    for general free discontinuity problems. Arch. Ration.  Mech. Anal. 211
    (2014), no. 2, 489-511.

    \bibitem{B4} D. Bucur, M. Nahon, C. Nitsch and C. Trombetti, Shape
    optimization of a thermal insulation problem. Preprint available at
    https://arxiv.org/abs/2112.07300

    \bibitem{B1} G. Buttazzo, An optimization problem for thin insulating
    layers around a conducting medium. Boundary control and boundary variations
    (Nice, 1986), 91-95, Lecture Notes in Comput. Sci., 100, Springer, Berlin,
    1988.

    \bibitem{CK} L. A. Caffarelli and D. Kriventsov, A free boundary problem
    related to thermal insulation. Comm. Partial Differential Equations 41
    (2016), no. 7, 1149-1182.

    \bibitem{C1} A. Chambolle, Convex representation for lower semicontinuous
    envelopes of functionals in $L^1$. J. Convex Anal. 8 (2001), no. 1,
    149-170.

    \bibitem{C3} A. Chambolle, D. Cremers and E. Strekalovskiy, A Convex
    Representation for the Vectorial Mumford-Shah Functional. 2012 IEEE
    Conference on Computer Vision and Pattern Recognition.

    \bibitem{C4} A. Chambolle, V. Duval, G. Peyr\'e and C. Poon, Geometric
    properties of solutions to the total variation denoising problem. Inverse
    Problems 33 (2017), no. 1, 015002, 44 pp.

    \bibitem{C5} A. Chambolle and M. Novaga. Anisotropic and crystalline mean
    curvature flow of mean-convex sets. Annali Scuola Normale Superiore -
    Classe di Scienze, p. 17, Mar. 2021.

    \bibitem{M1} G. Dal Maso, M. G. Mora and M. Morini, Local calibrations for
    minimizers of the Mumford-Shah functional with rectilinear discontinuity
    sets. J. Math. Pures Appl. (9) 79 (2000), no. 2, 141-162.

    \bibitem{DPS} T. De Pauw and D. Smets, On explicit solutions for the
    problem of Mumford and Shah. Commun. Contemp. Math. 1 (1999), no. 2,
    201-212.

    \bibitem{Fo} G. B. Folland, Introduction to partial differential equations.
    Second edition. Princeton University Press, Princeton, NJ, 1995. xii+324
    pp. ISBN: 0-691-04361-2

    \bibitem{F} N. Fusco, An overview of the Mumford-Shah problem. Milan J.
    Math. 71 (2003), 95-119.

    \bibitem{K} D. Kriventsov, A free boundary problem related to thermal
    insulation: flat implies smooth. Calc. Var. Partial Differential Equations
    58 (2019), no. 2, Paper No. 78, 83 pp.

    \bibitem{LM} C. Labourie and E. Milakis, Higher integrability of the
    gradient for the Thermal Insulation problem. Preprint available at
    arXiv:2101.09692.

    \bibitem{Maggi} F. Maggi,
    Sets of finite perimeter and geometric variational problems. An
    introduction to geometric measure theory. Cambridge Studies in Advanced
    Mathematics, 135. Cambridge University Press, Cambridge, 2012. xx+454 pp.
    ISBN: 978-1-107-02103-7

    \bibitem{M3} M. G. Mora, Local calibrations for minimizers of the
    Mumford-Shah functional with a triple junction. Commun. Contemp. Math. 4
    (2002), no. 2, 297-326.

    \bibitem{M2} M. G. Mora and M. Morini, Local calibrations for minimizers of
    the Mumford-Shah functional with a regular discontinuity set. Ann. Inst.
    H. Poincar\'eAnal. Non Linair\'e 18 (2001), no. 4, 403-436.

    \bibitem{M4} M. Morini, Global calibrations for the non-homogeneous
    Mumford-Shah functional. Ann. Sc. Norm. Super. Pisa Cl. Sci. (5) 1 (2002),
    no. 3, 603-648.

    \bibitem{bv_supersolutions} C. Scheven and T. Schmidt, BV supersolutions to
    equations of 1-Laplace and minimal surface type. J. Differential Equations
    261 (2016), no. 3, 1904-1932.


\end{thebibliography}
\end{document}